\documentclass[11pt]{article}
\usepackage{latexsym, amsmath, amssymb, amsthm}
\usepackage{enumerate, mathrsfs, bbm}
\usepackage{graphicx, tikz}
\usepackage{authblk, setspace, cite}
\usepackage[top = 1 in, bottom = 1 in, left = 1.2 in, right = 1.2 in]{geometry}

\newtheorem{theorem}{Theorem}[section]
\newtheorem{lemma}[theorem]{Lemma}
\newtheorem{corollary}[theorem]{Corollary}
\newtheorem{proposition}[theorem]{Proposition}

\newtheorem{definition}[theorem]{Definition}

\numberwithin{equation}{section}

\makeatletter
\def\blfootnote{\xdef\@thefnmark{}\@footnotetext}
\makeatother
 
\def\m{\mathbb}					
\def\lam{\lambda}    	\def\th{\theta}
\def\a{\alpha}	\def\b{\beta}	\def\ls{\lesssim}
\def\p{\partial}  
\def\wh{\widehat}	\def\la{\langle}	\def\ra{\rangle}
\def\ls{\lesssim}	\def\gs{\gtrsim}
		
\def\be{\begin{equation}}     \def\ee{\end{equation}}
\def\bp{\begin{pmatrix}}	\def\ep{\end{pmatrix}} 
\def\g{\gamma}		\def\F{\mathscr{F}}

\begin{document}
\title{Effect of lower order terms on the well-posedness of Majda-Biello systems 
}
\author[]{Xin Yang, Shenghao Li and Bing-Yu Zhang}
\date{}
\maketitle

\begin{abstract}

This paper investigates  a noteworthy phenomenon within the framework of Majda-Biello systems, wherein the inclusion of lower-order terms can enhance the well-posedness of the system. Specifically, we investigate the initial value problem (IVP) of the following system:
 \[
	\left\{
	\begin{array}{l}
		u_{t} + u_{xxx} = - v v_x,\\
		v_{t} + \alpha v_{xxx} + \beta v_x = - (uv)_{x},\\
		(u,v)|_{t=0} = (u_0,v_0) \in H^{s}(\mathbb{R}) \times H^{s}(\mathbb{R}),
	\end{array}
	\right. \quad x \in \mathbb{R}, \, t \in \mathbb{R},
\]
where $\alpha \in \mathbb{R}\setminus \{0\}$ and $\beta \in \mathbb{R}$. Let $s^{*}(\alpha, \beta)$  be the smallest value for 
which the IVP  is locally analytically well-posed in $H^{s}(\mathbb{R})\times H^{s}(\mathbb{R}) $ when $s > s^{}(\alpha, \beta)$. 
Two interesting facts have already been known in literature: $s^{*}(\alpha, 0) = 0$ for $\alpha \in (0,4)\setminus\{1\}$ and $s^*(4,0) = \frac34$.
Our key findings include the following:
\begin{itemize}
\item For $s^{*}(4,\beta)$, a significant reduction is observed,  reaching  $\frac12$ for $\beta > 0$  and $\frac14$ for $\beta < 0$.
\item Conversely, when $\alpha \neq 4$, we demonstrate that the value of $\beta$ exerts no influence on $s^*(\alpha, \beta)$.
\end{itemize}
These results shed light on the intriguing behavior of Majda-Biello systems when lower-order terms are introduced and provide valuable insights into the role of $\alpha $ and $\beta$ in the well-posedness of the system.

\end{abstract}


\blfootnote{\hspace{-0.25in} 2010 Mathematics Subject Classification. 35Q53; 35G55; 35L56; 35D30.}

\blfootnote{\hspace{-0.25in} Key words and phrases. KdV-KdV systems;   Majda-Biello systems; Local well-posedness; Fourier restriction spaces; Bilinear estimates.}

\begin{center}
\tableofcontents
\end{center}

\section{Introduction}  

\quad

The initial value problem (IVP) associated with the Korteweg-de Vries (KdV) equation,
\begin{equation} \label{x-1}
u_t + uu_x + \alpha u_{xxx} + \beta u_{x} = 0, \quad u(x,0) = \phi (x),
\end{equation}
has undergone extensive study for its well-posedness over the past six decades, dating back to the late 1960s (cf. \cite{Sjo67, Sjo70, Tem69, BS75, Kat83,KPV91JAMS, KPV93Duke,KPV93CPAM,Bou93b, KPV96, KT06, CCT03, CKSTT03, KV19,Mol11, Mol12} and the references therein).  Investigations have encompassed both the entire real line $\mathbb{R}$ and the torus $\mathbb{T}:=\m{R}/\m{Z}$, with a primary focus on establishing  the well-posedness within the classical Sobolev spaces $H^s (\mathbb{R})$ or $H^s (\mathbb{T})$.

It is now well-established that the IVP (\ref{x-1}) is analytically well-posed in the space $H^s (\mathbb{R})$ if and only if $s \geq -\frac{3}{4}$. Moreover, it is $C^0$ well-posed in $H^s (\mathbb{R})$ if and only if $s \geq -1$. The significance of the  index $-\frac{3}{4}$ is underscored by the fact that local well-posedness of the IVP (\ref{x-1}) in $H^s (\mathbb{R})$ can be established using the contraction mapping principle if and only if $s \geq -\frac{3}{4}$. 
This observation carries profound implications as we will illustrate next. When considering the local well-posedness of the IVP (\ref{x-1}) in $H^s (\mathbb{R})$ for $s \geq -\frac{3}{4}$, the associated linear IVP
\begin{equation} \label{x-2}
u_t + \alpha u_{xxx} + \beta u_{x} = 0, \quad u(x,0) = \phi (x),
\end{equation}
plays a pivotal role, where the nonlinear term $uu_x$ is treated as a small perturbation. However, for $s < -\frac{3}{4}$, the nonlinear term $uu_x$ can no longer be regarded as a small perturbation. In such cases, more sophisticated nonlinear analysis tools must be employed to establish the well-posedness of the IVP (\ref{x-1}) in $H^s (\mathbb{R})$. For this reason, $-\frac34 $ can be regarded as the  critical index of the IVP (\ref{x-1}) concerning the nonlinear effect. 
 
This paper investigates a modified version of the Majda-Biello systems described as follows:
\be\label{m-MB}
	\left\{
	\begin{array}{rcl}
		u_{t} + u_{xxx} & = & - v v_x,\\
		v_{t} + \a v_{xxx} + \b v_x & = & - (uv)_{x},\\
		(u,v)|_{t=0} & = & (u_0,v_0) \in \mathscr{H}^{s}(\m{R}),
	\end{array}
	\right. \quad x \in \m{R}, \, t \in \m{R}.
\ee
where $x$ and $t$ are real numbers, $\alpha \in \mathbb{R} \setminus \{0\}$ and $\beta \in \mathbb{R}$. The function space $\mathscr{H}^{s}(\mathbb{R})$ is defined as the Cartesian product $H^{s}(\mathbb{R}) \times H^{s}(\mathbb{R})$, where $H^{s}(\mathbb{R})$ denotes the standard Sobolev space. Let $s^{}(\alpha, \beta)$, referred to as the critical index, be the smallest value for which the IVP  (\ref{m-MB}) is locally analytically well-posed in $\mathscr{H}^{s}(\mathbb{R})$ when $s > s^{}(\alpha, \beta)$.  As in the case of the IVP (\ref{x-1}), this condition is equivalent to stating that the local well-posedness of the IVP (\ref{m-MB}) can be established using the contraction mapping principle. Oh in \cite{Oh09} first studied (\ref{m-MB}) with $\beta = 0$ and proved that $s^{*}(\alpha, 0) = 0$ when $\alpha \in (0, 4) \setminus \{1\}$. Moreover, Oh asserted that when $\alpha < 0$, $\alpha > 4$, or $\alpha = 1$, the problem (\ref{m-MB}) is locally analytically well-posed in $\mathscr{H}^s(\mathbb{R})$ for $s > -\frac{3}{4}$.
The special case when $\a = 4$ was later handled by Yang and Zhang in \cite{YZ22a}.


Typically, for a fixed $\alpha$, it is anticipated that the parameter $\beta$ does not significantly affect the value of $s^{*}$ since $\beta v_{x}$ is a lower-order term when compared to $\alpha v_{xxx}$.  This is the case for the  IVP  of the KdV equation,
\begin{equation}\label{KdV}
u_t + uu_x + \alpha u_{xxx} + \beta u_{x} = 0,
\qquad u(x,0) = u_0 (x) \in H^{s}(\mathbb{R}).
\end{equation}
Its critical index $-\frac34$ is irrespective of the values of $\alpha$ and $\beta$.


In contrast to this well-established fact, this paper will reveal an intriguing phenomenon: for a special $\alpha$, the inclusion of the lower-order term $\beta v_{x}$ in (\ref{m-MB}) will actually decrease the critical index $s^{*}$ compared to the case when $\beta=0$. On the one hand, when $\alpha = 4$ and $\beta = 0$, it has been previously demonstrated in \cite{YZ22a} that (\ref{m-MB}) exhibits analytical well-posedness in $\mathscr{H}^{s}(\mathbb{R})$ for all $s \geq \frac{3}{4}$. Conversely, employing a similar approach as in \cite{Bou93a} and \cite{Oh09}, it becomes evident that (\ref{m-MB}) lacks local $C^2$ well-posedness in $\mathscr{H}^{s}(\mathbb{R})$ for any $s < \frac{3}{4}$.
\footnote{\noindent For completeness, we also include a proof of this statement in the appendix.}
As a result, we have:
\be\label{old_ind}
	s^{*}(4,0) = \frac34. 
\ee
 On the other hand, this paper will explore the scenario in which $\alpha = 4$ and $\beta\neq 0$. It will demonstrate that the critical index  $s^{*}$ can be diminished, as shown in Theorem \ref{Thm, main_R} below, which serves as the core contribution of this paper. For the sake of brevity, we'll use the abbreviation ``LWP" for ``locally well-posed".

\begin{theorem}\label{Thm, main_R}
	Consider the modified Majda-Biello system (\ref{m-MB}) with $\alpha = 4$.
	\begin{itemize}
		\item[(a)] If $\beta > 0$, then (\ref{m-MB}) is analytically LWP in $\mathscr{H}^s(\mathbb{R})$ for all $s \geq \frac{1}{2}$, but it lacks $C^2$ LWP  in $\mathscr{H}^s(\mathbb{R})$ for all $s < \frac{1}{2}$.
		
		\item[(b)] If $\beta < 0$, then (\ref{m-MB}) is analytically LWP  in $\mathscr{H}^s(\mathbb{R})$ for all $s > \frac{1}{4}$, but it lacks $C^3$ LWP  in $\mathscr{H}^s(\mathbb{R})$ for all $s < \frac{1}{4}$.	
	\end{itemize}
	Therefore, the critical	 index  $s^{*}(4, \beta)$ is given by:
	\begin{equation}\label{}
		s^{*}(4, \beta) = 
		\begin{cases}
			\frac{1}{2} & \text{if } \beta > 0, \\
			\frac{1}{4} & \text{if } \beta < 0.
		\end{cases}
	\end{equation}
\end{theorem}

Motivated by the insights from Theorem \ref{Thm, main_R}, it becomes intriguing to explore the behavior of the parameter $\alpha$ for values beyond $\mathbb{R} \setminus \{0, 4\}$.

Based on Oh's work in \cite{Oh09}, Yang and Zhang in \cite{YZ22a} continued to explore the problem (\ref{m-MB}) with $\beta \neq 0$. Their investigation extended to the well-posedness problem of a broader class of coupled KdV-KdV systems, of which (\ref{m-MB}) represents a special case.
Let's revisit the general setup (1.2) presented in \cite{YZ22a}: 
\[
	\left\{\begin{array}{rcl}
		u_t + a_{1} u_{xxx} + b_{11} u_x & = & -b_{12} v_x + c_{11} uu_x + c_{12} vv_x + d_{11}u_{x} v + d_{12} uv_{x}, \vspace{0.03in}\\
		v_t + a_{2} v_{xxx} + b_{22} v_x & = & -b_{21} u_x + c_{21} uu_x + c_{22} vv_x + d_{21} u_{x}v + d_{22} uv_{x}, \vspace{0.03in}\\
		\left. (u,v)\right |_{t=0} & = & (u_{0}, v_{0}).
	\end{array}\right.
\]
Comparing this with our current model (\ref{m-MB}), the coefficients have the following specific values:
\[ 
	\begin{array}{c}
		a_1 = 1, \, a_2 = \a, \quad b_{11} = b_{12} = b_{21} = 0, \, b_{22} = \b, \\
		c_{11} = c_{21} = c_{22} = 0, \, c_{12} = -1, \quad d_{11} = d_{12} = 0, \, d_{21} = d_{22} = -1.
	\end{array}
\]
Applying Theorem 1.2 from \cite{YZ22a} to (\ref{m-MB}), we conclude that (\ref{m-MB}) is analytically LWP in $\mathscr{H}^{s}(\mathbb{R})$ in the following situations:
\begin{itemize}
	\item $\a \in (0,4) \setminus \{1\}$, $\b \in \m{R}$ and $s \geq 0$.
	
	\item $\a < 0$, $\a > 4$ or $\a = 1$, $\b \in \m{R}$ and $s > -\frac34$.
\end{itemize}

Now we will further demonstrate in Theorem \ref{Thm, auxi_R} that the above thresholds for $s$ are sharp, regardless of the value of $\beta$. This implies that the presence of the lower-order term $\beta v_x$ in (\ref{m-MB}) does not influence the analytical well-posedness thresholds of $s$ when $\alpha \notin \{0, 4\}$.

\begin{theorem}\label{Thm, auxi_R}
	Let $\a \in \m{R} \setminus \{0, 4\}$ in the modified Majda-Biello system (\ref{m-MB}).
	\begin{itemize}
		\item[(a)] If $\a \in (0,4) \setminus \{1\}$, then for any $\b \in \m{R}$, 
		(\ref{m-MB}) fails to be $C^2$ LWP in $\mathscr{H}^s(\m{R})$ for any $s < 0$.	
		
		\item[(b)] If $ \a < 0$, $\a > 4$ or $\a = 1$, then for any $\b \in \m{R}$, 
		(\ref{m-MB}) fails to be $C^3$ LWP in $\mathscr{H}^s(\m{R})$ for any $s < -\frac34$.	
	\end{itemize}
\end{theorem}

Having combined all the previously established results, we can present a relatively comprehensive overview of the well-posedness outcomes for (\ref{m-MB}).

\begin{corollary}\label{Cor, summary}
Consider $\alpha \in \mathbb{R} \setminus \{0\}$ in the modified Majda-Biello system (\ref{m-MB}). The regularity threshold $s^{*}$ for (\ref{m-MB}), concerning analytical well-posedness, is characterized as follows:

\[  
	s^{*}(\a, \b) = \left\{
	\begin{array}{cll}
		3/4 & \text{if} & \a = 4, \, \b = 0, \\
		1/2 & \text{if} & \a = 4, \, \b > 0, \\
		1/4 & \text{if} & \a = 4, \, \b < 0, \\
		0 & \text{if} & \a \in (0,4) \setminus \{1\} \text{ and } \b \in \m{R}, \\
		-3/4 & \text{if} & \a < 0, \a > 4 \text{ or } \a = 1, \text{ and } \b \in \m{R}.			
	\end{array}
	\right.
\]
\end{corollary}


Up until Corollary \ref{Cor, summary}, we have introduced linear terms exclusively to the $v$-equation in (\ref{m-MB}). However, following a similar approach, we can include linear terms in both the $u$-equation and the $v$-equation of (\ref{m-MB}). This leads us to consider the following modified Majda-Biello system:
\be\label{gm-MB}
\left\{
\begin{array}{rcl}
	u_{t} + u_{xxx} + \b_1 u_{x} & = & - v v_x,\\
	v_{t} + \a v_{xxx} + \b_2 v_x & = & - (uv)_{x},\\
	(u,v)|_{t=0} & = & (u_0,v_0) \in \mathscr{H}^{s}(\m{R}),
\end{array}
\right. \quad x \in \m{R}, \, t \in \m{R},
\ee
where $\a \in \m{R}\setminus\{0\}$, $\b_1, \b_2\in\m{R}$. 

Using a similar argument as the proofs of Theorem \ref{Thm, main_R} and Theorem \ref{Thm, auxi_R}, one can find that the difference $\beta_2 - \beta_1$ plays an analogous role in (\ref{gm-MB}) to that of $\beta$ in (\ref{m-MB}). We denote $s^{*}_1(\alpha, \beta_1, \beta_2)$ as the regularity threshold for analytical LWP for the systems (\ref{gm-MB}). Consequently, we arrive at the following result:

\begin{corollary}\label{Cor, wp for gm-MB}
Let $\a \in \m{R} \setminus \{0\}$ and $\b_1, \b_2 \in \m{R} $ in (\ref{gm-MB}).  The regularity threshold $s^{*}$ for (\ref{gm-MB}), in terms of analytical well-posedness, is determined as follows:
\[  
s^{*}(\a, \b_1, \b_2) = \left\{
\begin{array}{cll}
	3/4 & \text{if} & \a = 4, \, \b_2 = \b_1, \\
	1/2 & \text{if} & \a = 4, \, \b_2 > \b_1, \\
	1/4 & \text{if} & \a = 4, \, \b_2 < \b_1, \\
	0 & \text{if} & \a \in (0,4) \setminus \{1\} \text{ and } \b_1, \b_2 \in \m{R}, \\
	-3/4 & \text{if} & \a < 0, \a > 4 \text{ or } \a = 1, \text{ and } \b_1, \b_2 \in \m{R}.			
\end{array}
\right.
\]
\end{corollary}

As we mentioned earlier, Majda-Biello systems are a specialized class within more general coupled KdV-KdV systems. In addition to Majda-Biello systems, two other widely studied systems are the Hirota-Satsuma systems (see, for instance, \cite{HS81, AC08, Fen94}) and the Gear-Grimshaw systems (see, for example, \cite{GG84, AC08, ACW96, BPST92, LP04, ST00}). It would be a natural future endeavor to investigate similar phenomena for these two types of coupled KdV-KdV systems.

Now, we wish to highlight some key ideas in the proofs of Theorem \ref{Thm, main_R} and \ref{Thm, auxi_R}, particularly explaining why the case $\alpha = 4$ is of special significance. Originating from seminal works such as \cite{Bou93a, KPV96}, where Fourier restriction spaces were introduced, the method of bilinear (multilinear) estimates on these spaces has become a standard approach for studying well-posedness in dispersive equations and systems. As explicitly pointed out by Tao in \cite{Tao01}, a crucial element in bilinear estimates is the associated resonance function. For example, in the case of our target problem (\ref{m-MB}), one of the bilinear estimates (see (\ref{bilin-R2})) takes the form:

\[
	\| (uv)_{x} \|_{X^{\a, \b}_{s, b-1}( \m{R}^2 )} \leq C \| u \|_{X^{1}_{s, b}( \m{R}^2 )} \| v \|_{X^{\a, \b}_{s, b}( \m{R}^2 )},  \quad \forall\, u, v \in \mathscr{S}(\m{R}^{2}),
\]
whose associated resonance function is 
\[
	G_0(\eta_1, \eta_2, \eta_3) \triangleq \eta_1^3 + (\a \eta_2^3 - \b \eta_2) + (\a \eta_3^3 - \b \eta_3), \quad \forall\, \sum_{i=1}^{3} \eta_i = 0.
\]
Denote $\langle G_0 \rangle := 1 + |G_0|$. Then in a heuristic sense, the larger $\langle G_0 \rangle$ is, the smaller the threshold $s^*$ will be. Therefore, we need to investigate how small $\langle G_0 \rangle$ can become and the extent of the region where $\langle G_0 \rangle$ is small.
By replacing $\eta_3$ with $-(\eta_1 + \eta_2)$,  we derive the expression for $G_0$ as follows: 
\be\label{G0intro}
	G_0(\eta_1, \eta_2, \eta_3) = -3\a \eta_1^3 f(\eta_2 / \eta_1) + \b\eta_1,
\ee
where 
\[ 
	f(x) = x^2 + x + \frac{\a - 1}{3\a}.
\] 
For the special case when $\alpha = 4$, the function $f(x)$ possesses a repeated root at $-\frac{1}{2}$, indicating a significant resonance effect. In this situation, $G_0$ can be expressed as:

\[G_0(\eta_1, \eta_2, \eta_3) = -3\eta_1 \big[ (2\eta_2 + \eta_1)^2 - \b/3 \big].\]
\begin{itemize}
	\item If $\beta = 0$, then $G_0(\eta_1, \eta_2, \eta_3) = -3\eta_1 (2\eta_2 + \eta_1)^2$. Consequently, for any sufficiently large number $N$, $\langle G_0 \rangle \sim 1$ in the region $D := \{\eta_1 \sim N, , 2\eta_2 + \eta_1 \sim N^{-1/2}\}$.
On the one hand, $\langle G_0 \rangle$ is ideally small in $D$. On the other hand, the volume of $D$ is large enough to determine the threshold $s^* = \frac34$. 
	
	\item If $\beta > 0$, then by introducing $\beta_1 = \sqrt{\beta/3}$, $G_0$ can be further factored as follows:
	\[
		G_0(\eta_1, \eta_2, \eta_3) = -3 \eta_1 (2\eta_2 + \eta_1 + \b_1)(2\eta_2 + \eta_1 - \b_1).
	\]
	Since $\beta_1 \neq 0$, at most one of the terms $2\eta_2 + \eta_1 + \beta_1$ and $2\eta_2 + \eta_1 - \beta_1$ can fall below a constant level. This contrasts with the case of $\beta = 0$, where the term $2\eta_2 + \eta_1$ has a power of 2 in $G_0$. Therefore, in this case, the threshold $s^*$ can be slightly smaller.
	
	\item If $\beta < 0$, then regardless of the relationship between $\eta_1$ and $\eta_2$, the magnitude of $\langle G_0 \rangle$ is at least $N$ when $\eta_1 \sim N$. This large lower bound for $\langle G_0 \rangle$ can significantly reduce the threshold $s^*$ to $\frac14$.
	
\end{itemize}
When $\a \neq 4$, the function $f$ either has no root or has two distinct roots. As a result, the influence of the linear term $\beta \eta_1$ in (\ref{G0intro}) is overshadowed by the cubic term $\eta_1^3 f(\eta_2 / \eta_1)$. This explains why the value of $\beta$ does not affect the threshold $s^*$ when $\a \neq 4$.

The remainder of this paper is structured as follows. In Section \ref{Sec, bilin}, we establish crucial bilinear estimates that can be utilized to demonstrate analytical well-posedness when $s > s^*$ in Theorem \ref{Thm, main_R}. Subsequently, in Section \ref{Sec, pf-main}, we complete the proof for Theorem \ref{Thm, main_R} by establishing the ill-posedness when $s < s^*$. Finally, we confirm Theorem \ref{Thm, auxi_R} in Section \ref{Sec, pf-auxi}.

\section{Bilinear Estimates}
\label{Sec, bilin}

\subsection{Preliminaries}
\label{Subsec, Prelim}
To establish the well-posedness component in Theorem \ref{Thm, main_R}, it is sufficient to establish bilinear estimates within the Fourier restriction space, as detailed in Proposition \ref{Prop, bilin-R} below. This method, originally introduced by Bourgain \cite{Bou93a, Bou93b} and Kenig-Ponce-Vega \cite{KPV96} for studying the KdV equation, involves the definition of the Fourier restriction space.

Let $\alpha, \beta \in \mathbb{R}$ with $\alpha \neq 0$, and denote the polynomial $\phi^{\alpha,\beta}$ as
\be\label{phase fn}
\phi^{\alpha, \beta}(\xi) = \alpha \xi^{3} - \beta \xi, \quad\forall, \xi \in \mathbb{R}.
\ee
For convenience, $\phi^{\alpha,0}$ is denoted as $\phi^{\alpha}$. The Fourier restriction space is then defined as follows.

\begin{definition}\label{Def, FR space on R}
	For any $\alpha, \beta, s, b \in \mathbb{R}$ with $\alpha \neq 0$, the Fourier restriction space $X^{\alpha,\beta}{s, b}(\mathbb{R}^2)$ is defined as the completion of the Schwartz space $\mathscr{S}(\mathbb{R}^{2})$ with the norm
	\be\label{FR norm on R}
		\|w\|_{ X^{\a,\b}_{s,b}( \m{R}^2 ) } = \|\la \xi \ra ^{s} \la \tau - \phi^{\a,\b}(\xi) \ra^{b} \wh{w}(\xi,\tau)\|_{ L^{2}_{\xi\tau} (\m{R}^2) },
	\ee
	where $\la \cdot \ra = 1 + |\cdot|$, $\phi^{\a,\b}$ is given by (\ref{phase fn}), and $\wh{w}$ refers to the space-time Fourier transform of $w$. Furthermore, $X^{\alpha,0}_{s,b}(\mathbb{R}^2)$ is abbreviated as $X^{\alpha}_{s,b}(\mathbb{R}^2)$.
\end{definition}

Similar to the KdV equation, (\ref{m-MB}) exhibits sub-critical behavior when $s > -\frac32$. Exploiting this characteristic, one can rescale the equation, making the initial data small and facilitating the establishment of well-posedness. Following the approach in \cite{KPV96, YZ22a}, we introduce the functions $u^{\lambda}$ and $v^{\lambda}$ for $\lambda \geq 1$ as follows:
\be\label{scaled fn}
	\left\{
	\begin{aligned}
		u^{\lam}(x,t) = \lam^{-2} u(\lam^{-1} x, \lam^{-3} t), \\
		v^{\lam}(x,t) = \lam^{-2} v(\lam^{-1} x, \lam^{-3} t), 
	\end{aligned}
	\qquad x\in\m{R}, \,t\in\m{R}.
	\right.
\ee
Then (\ref{m-MB}) can be reformulated as the following system:
\be\label{m-MB-scaled}
	\left\{
	\begin{array}{rcl}
		u^{\lam}_{t} + u^{\lam}_{xxx} &=& - v^{\lam} v^{\lam}_x,\\
		v^{\lam}_{t} + \a v^{\lam}_{xxx} + \lam^{-2} \b v^{\lam}_x &=& - (u^{\lam} v^{\lam})_{x},\\
		( u^{\lam}, v^{\lam} )|_{t=0} &=& (u^{\lam}_0, v^{\lam}_0) \in \mathscr{H}^{s}(\m{R}),
	\end{array}
	\right. \qquad x \in \m{R}, \, t \in \m{R}.
\ee
where $	\big( u^{\lam}_{0}(x), v^{\lam}_{0}(x) \big) = \lam^{-2} \big( u_{0}(\lam^{-1} x), v_{0}(\lam^{-1} x)\big)$.
Given that $\lambda \geq 1$ and $s \geq 0$, it follows that
\[
	\|u^{\lam}_{0}\|_{ H^{s}(\m{R}) } \leq \lam^{-\frac32} \|u_{0}\|_{ H^{s}(\m{R}) }, \quad 
	\|v^{\lam}_{0}\|_{ H^{s}(\m{R}) } \leq \lam^{-\frac32} \|v_{0}\|_{ H^{s}(\m{R}) }.
\]
Consequently, as $\lambda \to \infty$, both the size of the initial data and the coefficient of the lower-order term $v^{\lambda}{x}$ decay to 0, i.e.,
\be\label{small coef}
	\lim_{\lam \to 0} \lam^{-2} |\b| = 0, \quad 
	\lim_{\lam \to 0} \| (u^{\lam}_0, v^{\lam}_0) \|_{\mathscr{H}^s(\m{R})} = 0.
\ee
With (\ref{small coef}), we can assume the coefficient of the first-order terms to be small, reflected in the condition $0< |\beta| \leq 1$ in the following proposition.

\begin{proposition}\label{Prop, bilin-R}
Consider $\alpha = 4$ and assume one of the two conditions below:

	\begin{itemize}
		\item[(i)] $0< \b \leq 1$, $s \geq \frac12$ and $\frac12 < b \leq \frac34$;
		\item[(ii)] $-1 \leq \b < 0$, $s > \frac14$ and $\frac12 < b \leq s + \frac14$.
	\end{itemize}
	Then there exists a constant $C = C(s, b)$ such that the following bilinear estimates hold.
	\begin{eqnarray}
		\| v v_{x} \|_{X^{1}_{s, b- 1}( \m{R}^2 )} &\leq & C \| v \|_{X^{\a, \b}_{s, b}( \m{R}^2 )}^2, 
		\label{bilin-R1} \\
		\| (uv)_{x} \|_{X^{\a, \b}_{s, b-1}( \m{R}^2 )} & \leq & C \| u \|_{X^{1}_{s, b}( \m{R}^2 )} \| v \|_{X^{\a, \b}_{s, b}( \m{R}^2 )}, 
		\label{bilin-R2}		
	\end{eqnarray}
where $u, v \in \mathscr{S}(\m{R}^{2})$.
\end{proposition}
To establish the aforementioned bilinear estimates, we require some elementary auxiliary lemmas.
\begin{lemma}\label{Lemma, int in tau}
	Let $\rho_{1}>1$ and $0\leq\rho_{2}\leq\rho_{1}$ be given. There exists  a constant $C=C(\rho_{1},\rho_{2})$ such that for any $\a,\b\in\m{R}$,
	\be\label{int in tau}
		\int_{-\infty}^{\infty} \frac{dx}{\la x - \a \ra^{\rho_{1}} \la -x - \b \ra^{\rho_{2}} } \leq \frac{C}{\la \a + \b \ra^{\rho_{2}} }.
	\ee
\end{lemma}
The proof for this lemma is standard and therefore omitted. It is worth noting that $\langle \alpha + \beta \rangle = \langle (x - \alpha) + (-x - \beta) \rangle$, an observation that will be utilized in subsequent estimates.

\begin{lemma}\label{Lemma, bdd int}
	If $\rho > \frac{1}{2}$, then there exists $C = C(\rho)$ such that for any $\sigma_{i} \in \mathbb{R}$, $0 \leq i \leq 2$, with $\sigma_{2} \neq 0$,
	\be\label{bdd int for quad}
		\int_{-\infty}^{\infty} \frac{dx}{\la \sigma_{2}x^{2} + \sigma_{1} x + \sigma_{0} \ra^{\rho}}\leq \frac{C}{|\sigma_{2}|^{1/2}}.
	\ee
	Similarly, if $\rho > \frac{1}{3}$, then there exists $C = C(\rho)$ such that for any $\sigma_{i} \in \mathbb{R}$, $0 \leq i \leq 3$, with $\sigma_{3} \neq 0$,
	\be\label{bdd int for cubic}
		\int_{-\infty}^{\infty} \frac{dx}{\la \sigma_{3}x^{3} + \sigma_{2}x^{2} + \sigma_{1} x  + \sigma_{0} \ra^{\rho} }\leq \frac{C}{|\sigma_{3}|^{1/3}}.
	\ee
\end{lemma}

\begin{proof}
	We refer the reader to the proof of Lemma 2.5 in \cite{BOP97} where (\ref{bdd int for cubic}) was  proved. The similar argument can also be applied to  obtain (\ref{bdd int for quad}).
\end{proof}

If the power $\rho$ in Lemma \ref{Lemma, bdd int} is greater than 1, then stronger estimates hold, as indicated by Lemma \ref{Lemma, int for quad} and Lemma \ref{Lemma, int for cubic} below, the proofs of which can be found in \cite{YZ22a}.

\begin{lemma}\label{Lemma, int for quad}
	Let $\rho>1$ be given. There exists  a constant $C=C(\rho)$ such that for any $\sigma_{i}\in\m{R},\,0\leq i\leq 2$, with $\sigma_{2}\neq 0$, 
	\be\label{int for quad}
		\int_{-\infty}^{\infty} \frac{dx}{\la \sigma_{2}x^{2} + \sigma_{1} x + \sigma_{0} \ra^{\rho} } \leq C\, |\sigma_{2}|^{-\frac12} \Big\la \sigma_{0} - \frac{\sigma_{1}^{2}}{4\sigma_{2}} \Big\ra^{-\frac12}.
	\ee
\end{lemma}

\begin{lemma}\label{Lemma, int for cubic}
	Let $\rho>1$ be given. There exists a constant $C=C(\rho)$ such that  for any $\sigma_{i}\in\m{R},\,0\leq i\leq 2$,
	\be\label{int for cubic}
		\int_{-\infty}^{\infty} \frac{dx}{\la x^3 + \sigma_{2} x^{2} + \sigma_{1} x + \sigma_{0} \ra^{\rho}} \leq C \big \la 3\sigma_{1} - \sigma_{2}^{2} \big \ra^{-\frac14}.
	\ee
\end{lemma}

For the proof of the bilinear estimate, it is often advantageous to transform it into an estimate involving a weighted convolution of $L^2$ functions, as emphasized in \cite{Tao01, CKSTT03}. The following lemma provides such an example for a general bilinear estimate, and its proof follows standard techniques utilizing duality and the Plancherel theorem. To streamline notation, we use $\vec{\xi}=(\xi_{1},\xi_{2},\xi_{3})$ and $\vec{\tau}=(\tau_{1},\tau_{2},\tau_{3})$ to represent vectors in $\mathbb{R}^{3}$. Additionally, we define
\be\label{int domain}
	A := \Big\{ (\vec{\xi}, \vec{\tau}) \in \m{R}^{6}: \sum_{i=1}^{3} \xi_{i} = \sum_{i=1}^{3} \tau_{i} = 0 \Big\}.
\ee

\begin{lemma}\label{Lemma, bilin to weighted l2}
	Given $s$, $b$ and $\{(\a_{i},\b_{i})\}_{ 1\leq i\leq 3 }$,  the bilinear estimate
	\[\|\p_{x}(w_{1}w_{2})\|_{X^{\a_{3},\b_{3}}_{s,b-1}}\leq C\,\|w_{1}\|_{X^{\a_{1},\b_{1}}_{s,b}}\,\|w_{2}\|_{X^{\a_{2},\b_{2}}_{s,b}}, \quad\forall\, w_1, w_2 \in \mathscr{S}(\m{R}^{2}),\]
	is equivalent to
	\be\label{weighted l2 form, nd1}
		\int \limits_{A} \frac{\xi_{3} \la \xi_{3} \ra^{s} \prod \limits_{i=1}^{3} f_{i}(\xi_{i}, \tau_{i})}{\la\xi_{1}\ra^{s} \la\xi_{2}\ra^{s} \la L_{1}\ra^{b} \la L_{2}\ra^{b} \la L_{3}\ra^{1-b}} \leq C\, \prod_{i=1}^{3} \|f_{i}\|_{L^{2}_{\xi \tau}}, \quad \forall\, f_1, f_2, f_3 \in \mathscr{S}(\m{R}^{2}),
	\ee
	where 
	\be\label{def of L}
		L_{i} = \tau_{i} - \phi^{\a_{i}, \b_{i}}(\xi_{i}), \quad i = 1,2,3.
	\ee
\end{lemma}

Now we are ready to carry out the proof of Proposition \ref{Prop, bilin-R}.

\subsection{Proof of (\ref{bilin-R1}) in Proposition \ref{Prop, bilin-R}}
\label{Subsec, bilin-1}

According to Lemma \ref{Lemma, bilin to weighted l2}, the bilinear estimate (\ref{bilin-R1}) is equivalent to the following estimate:
\be\label{bilin1-wl}
\int_A \frac{ \xi_3 \la \xi_3 \ra^s \prod_{i=1}^3 f_i ( \xi_i, \tau_i ) }{ \la \xi_1\ra^s \la \xi_2 \ra^s \la L_1 \ra^b \la L_2 \ra^b \la L_3\ra^{1-b} } 
\leq C \prod_{i=1}^3 \| f_i \|_{L^2}, \quad \forall f_1, f_2, f_3 \in \mathscr{S}(\m{R}^{2}),
\ee
where
\be\label{bilin1-L}
L_1 = \tau_1 - \phi^{\a, \b} (\xi_1), \quad L_2 = \tau_2 - \phi^{\a, \b} (\xi_2), \quad L_3 = \tau_3 - \xi_3^3.
\ee
It is well-known that the resonance function $H$ plays an essential role in the bilinear estimate, so we first compute $H$. For any $(\vec{\xi}, \vec{\tau}) \in A$, where $A$ is as defined in (\ref{int domain}),
\[
	H(\vec{\xi}, \vec{\tau}) = \sum_{i=1}^{3} L_i = - \big( \phi^{\a, \b} (\xi_1) +  \phi^{\a, \b} (\xi_2) + \xi_3^3 \big).
\]
When $\a = 4$, the above expression can be simplified as 
\be\label{res fun}
	H(\vec{\xi}, \vec{\tau}) = 3\xi_3 \Big[ (\xi_2 - \xi_1)^2 - \frac{\b}{3} \Big].
\ee

\subsubsection{Condition (i)}
\label{Subsubsec, bl1-bp}

We first deal with the case when $0< \b \leq 1$, $s \geq \frac12$ and $\frac12 < b \leq \frac34$. Based on (\ref{bilin1-wl}), we split the integral domain $A$ into several pieces so that the integral on each piece can be controlled easier.

\begin{itemize}
	\item Case 1: $ | \xi_3 | \ls 1$.
	
	Since $s>0$ and $\la \xi_3 \ra \leq \la \xi_1 \ra \la \xi_2 \ra$, it suffices to prove
	\[
	\int \frac{ | \xi_3 | \prod_{i=1}^3 | f_i (\xi_i, \tau_i) |}{ \la L_1\ra^b \la L_2\ra^b \la L_3 \ra^{1-b} } \leq C \prod_{i=1}^3 \| f_i \|_{L^2},
	\]
	i.e.
	\be\label{bl1-x3s}
	\iint \frac{ | f_3 | | \xi_3 |}{\la L_3 \ra^{1-b}}\bigg( \iint \frac{ | f_1  f_2 |}{\la L_1 \ra^b \la L_2\ra^b} \,d \xi_2 \, d \tau_2 \bigg) \,d \xi_3 \, d \tau_3 \leq C \prod_{i=1}^3 \| f_i \|_{L^2}.
	\ee
	
	By following the same argument as in \cite{KPV96} via the Cauchy-Schwarz inequality, it reduces to show
	\be\label{bl1-x3s-sup3-d}
	\sup _{ | \xi_3 | \ls 1, \tau_3 \in \m{R} } \frac{ | \xi_3 |^2}{\la L_3 \ra^{2 - 2b}} \iint \frac{ d \tau_2 \,d \xi_2}{\la L_1 \ra^{2b} \la L_2 \ra^{2 b}} \leq C.
	\ee
	Since
	$$
	L_1 = \tau_1 - \phi^{\a, \b} (\xi_1) = - \tau_2 - \tau_3 - \phi^{\a, \b}(- \xi_2 - \xi_3)
	$$
	and $ L_2 = \tau_2 - \phi^{\a, \b} (\xi_2) $, 
	it then follows from Lemma \ref{Lemma, int in tau} that
	\[
	\int \frac{d \tau_2}{\la L_1 \ra^{2b} \la L_2 \ra^{2b}} \leq \frac{C}{\la L_1 + L_2 \ra^{2b}}.
	\]
	
	Thus, (\ref{bl1-x3s-sup3-d}) boils down to
	\be\label{bl1-x3s-sup3}
	\sup _{ | \xi_3 | \ls 1, \tau_3 \in \m{R}} \frac{| \xi_3 |^2}{\la L_3 \ra^{2 - 2b}} \int \frac{d \xi_2}{\la L_1 + L_2 \ra^{2b}} \leq C.
	\ee
	By direct computation, for fixed $\xi_3$ and $\tau_3$, $ L_1 + L_2 $ can be viewed as a quadratic function of $\xi_2$. More precisely, when $\a = 4$,
	\be\label{L1aL2}
	L_1 + L_2 = 12 \xi_3 \xi_2^2 + 12 \xi_3^2 \xi_2 + 4 \xi_3^3 - \b \xi_3 - \tau_3 \stackrel{\Delta}{=} P_{\xi_3, \tau_3}(\xi_2).
	\ee
	Hence, we can take advantage of Lemma \ref{Lemma, int for quad} to conclude that
	$$
	\int \frac{d \xi_2}{\la L_1 + L_2 \ra^{2b}} \leq C | \xi_3 |^{ -\frac12 }.
	$$
	Then
	$$
	\text{ LHS of } ( \ref{bl1-x3s-sup3} ) \ls \sup _{ | \xi_3 | \ls 1, \tau_3 \in \m{R} } \frac{ | \xi_3 |^{\frac32} }{\la L_3 \ra^{2 - 2b}} \leq C.
	$$
	
	\item Case 2: $| \xi_3 | \gg 1$ and $ | \xi_2 - \xi_1 | \geq \frac{1}{10} | \xi_3 |$.
	
	Similar to Case 1, it reduces to show
	\be\label{bl1-x3l-dl-sup3}
		\sup _{|\xi_3| \gg 1, \tau_3\in\m{R}} \frac{ | \xi_3 |^2}{\la L_3 \ra^{2 - 2b}} \int \frac{d \xi_2}{\la L_1 + L_2 \ra^{2b} } \leq C.
	\ee
	Recalling that $ L_1 + L_2 = P_{\xi_3, \tau_3} (\xi_2)$ as in (\ref{L1aL2}), so
	\[
		P_{\xi_3, \tau_3}^{\prime} (\xi_2) = 24 \xi_3 \xi_2 + 12 \xi_3^2 = 12 \xi_3 (2 \xi_2 + \xi_3) = 12 \xi_3 (\xi_2 - \xi_1).
	\]
	Thanks to the assumption $|\xi_2 - \xi_1| \geq | \xi_3 | / 10$, we know
	$ | P_{\xi_3, \tau_3}^{\prime} (\xi_2) | \gs | \xi_3 |^2 $. Consequently,
	\[
	\text { LHS of } (\ref{bl1-x3l-dl-sup3}) \leq \sup _{|\xi_3| \gg 1, \tau_3 \in \m{R} } |\xi_3|^2 \int \frac{1}{|\xi_3|^2} \frac{|P_{\xi_3, \tau_3}^{\prime} (\xi_2)|}{\la P_{\xi_3, \tau_3}(\xi_2) \ra^{2b} } \,d \xi_2 \leq C.
	\]
	
	\item Case 3: $ |\xi_3| \gg 1$ and $|\xi_2 - \xi_1| < \frac{1}{10} | \xi_3 |$.
	
	In this case, it is not difficult to find that $| \xi_1 | \sim |\xi_2| \sim |\xi_3|$, so in order to justify (\ref{bilin1-wl}), it suffices to show
	\be\label{bl1-wl-as}
	\int \frac{ |\xi_3|^{1-s} \prod_{i=1}^3 |f_i|}{\la L_1\ra^b \la L_2\ra^b \la L_3 \ra^{1-b} } \leq C \prod_{i=1}^3 \| f_i \|_{L^2}.
	\ee
	Analogous to Case 1, it boils down to verify
	\be\label{bl1-x3l-ds-sup3}
		\sup _{ | \xi_3 | \gg 1, \tau_3 \in \m{R} } \frac{ |\xi_3|^{2 - 2s}}{\la L_3 \ra^{2 - 2b} } \int \frac{1}{\la L_1 + L_2\ra^{2b} } \,d \xi_2 \leq C.
	\ee
	Based on (\ref{L1aL2}),
	$$
	L_1 + L_2 = P_{\xi_3, \tau_3} (\xi_2) = \sigma_2 \xi_2^2 + \sigma_1 \xi_1 + \sigma_0,
	$$
	where
	$$
	\sigma_0 = 4 \xi_3^3 - \b \xi_3 - \tau_3, \quad \sigma_1 = 12 \xi_3^2, \quad \sigma_2 = 12 \xi_3.
	$$
	According to Lemma \ref{Lemma, int for quad},
	$$
	\begin{aligned}
		\int \frac{1}{\la L_1 + L_2 \ra^{2b} } \,d \xi_2 & \leq C |\sigma_2|^{-\frac12} \Big\la \sigma_0 - \frac{\sigma_1^2}{4 \sigma_2}\Big\ra^{-\frac12} \\
		& = C |\xi_3|^{-\frac12}\la \xi_3^3 - \b \xi_3 - \tau_3\ra^{-\frac12}.
	\end{aligned}
	$$
	Thus,
	\be\label{bl1-x3l-ds-qe}
	\text{ LHS of } (\ref{bl1-x3l-ds-sup3}) \leq C \sup _{|\xi_3| \gg 1, \tau_3 \in \m{R} } \frac{ |\xi_3|^{\frac32 - 2s} }{\la L_3\ra^{2 - 2b} \la \xi_3^3 - \b \xi_3 - \tau_3 \ra^{1/2} }.
	\ee
	Since $s \geq\frac12$ and $b \leq \frac34$,
	\be\label{bl1-x3l-ds-de}
	\text{ RHS of } (\ref{bl1-x3l-ds-qe}) \leq C \sup _{ |\xi_3| \gg 1, \tau_3 \in \m{R} } \frac{|\xi_3|^{\frac12} }{\big[ \la L_3 \ra \la \xi_3^3 - \b \xi_3 - \tau_3 \ra \big]^{1/2}}.
	\ee
	Noticing that
	$$
		\la L_3 \ra \la \xi_3^3 - \b \xi_3 - \tau_3\ra \geq \la L_3 + \xi_3^3 - \b \xi_3 - \tau_3 \ra = \la \b \xi_3 \ra,
	$$
	so
	$$
	\text { RHS of }(\ref{bl1-x3l-ds-de}) \leq C.
	$$
\end{itemize}

Hence, we finished the proof for the bilinear estimate (\ref{bilin-R1}) under condition (i) in Proposition \ref{Prop, bilin-R}.

\subsubsection{Condition (ii)}
\label{Subsubsec, bl1-bn}

Now we continue to investigate the case when $-1 \leq \b < 0$, $s > \frac14$ and $\frac12 < b \leq s + \frac14$. Again based (2.16), we split the region A into three pieces in the same way as above.
For Case 1: $ |\xi_3| \ls 1$, and Case 2: $|\xi_3| \gg 1$ and $|\xi_2 - \xi_1| \geq \frac{1}{10} |\xi_3|$,
the argument is exactly the same as those in Section \ref{Subsubsec, bl1-bp}.

Next, we focus on Case 3: $|\xi_3| \gg 1$ and $|\xi_2 - \xi_1| < \frac{1}{10} |\xi_3|$. 
In this case, $ |\xi_1| \sim |\xi_2| \sim |\xi_3|$ and it suffices to verify (\ref{bl1-wl-as}), i.e.
\be\label{bl1-bn-wl}
	\int \frac{ |\xi_3|^{1-s} \prod_{i=1}^3 |f_i|}{\la L_1 \ra^b \la L_2\ra^b \la L_3\ra^{1-b} } \leq C \prod_{i=1}^3 \| f_i \|_{L^2}, \quad \forall f_1, f_2, f_3 \in \mathscr{S} (\m{R}^2).
\ee
Since $-1 \leq \b<0$, there exists $\b_2>0$ such that 
$\b_2^2 = -\frac{\b}{3}$.  
As a result, it follows from (\ref{res fun}) that
\be\label{bl1-bn-Hlb}
	H(\vec{\xi}, \vec{\tau}) = 3 \xi_3 \big[ (\xi_2 - \xi_1)^2 + \b_2^2 \big] \gs |\xi_3|.
\ee
Since $H = \sum_{i=1}^3 L_i$, we know at least one of $\left\{\la L_i\ra \right\}_{1 \leq i \leq 3} $ is comparable to $\la H \ra$. Denote
$$
	\text{MAX} = \max \left\{ \la L_1\ra, \la L_2\ra, \la L_3\ra \right\}.
$$
Then $\text{MAX} \gs \la H \ra \gs |\xi_3|$. 
Next, we will further divide Case 3 into three subcases depending on which $\la L_i\ra$ equals $\text{MAX}$.

\begin{itemize}
	\item Case 3.1: $\la L_3\ra = \text{MAX}$. 
	
	In this case, $\la L_3\ra \gs |H| \gs |\xi_3|$. Then according to (\ref{bl1-wl-as}), it suffices to prove
	\be\label{bl1-bn-x3l-ds-sup3}
	\sup _{ |\xi_3| \gg 1, \tau_3 \in \m{R} } \frac{ |\xi_3|^{2 - 2s}}{ |\xi_3|^{2 - 2b} } \int \frac{1}{\la L_1 + L_2 \ra^{2b}} \,d \xi_2 \ls 1. 
	\ee
	Recalling
	$$
		L_1 + L_2 = P_{\xi_3, \tau_3} (\xi_2) = 12 \xi_3 \xi_2^2 + 12 \xi_3^2 \xi_2 + (4 \xi_3^3 - \b \xi_3 - \tau_3).
	$$
	
	Then we can apply Lemma \ref{Lemma, int for quad} to obtain
	$$
		\int \frac{1}{\la L_1 + L_2\ra^{2b} } \,d \xi_2 \ls | \xi_3 |^{-\frac12}.
	$$
	Hence,
	$$\begin{aligned}
		\text { LHS of } (\ref{bl1-bn-x3l-ds-sup3}) 
		&\ls \sup _{ |\xi_3| \gg 1, \tau_3 \in \m{R} } \frac{ |\xi_3 |^{2 - 2s} }{ |\xi_3|^{2 - 2b} } \cdot |\xi_3|^{-\frac12} \\
		& = \sup _{ |\xi_3| \gg 1, \tau_3 \in \m{R} } |\xi_3|^{2b - 2s - \frac12} \ls 1,
	\end{aligned}$$
	where the last inequality is due to the assumption that 
	$b \leq s + \frac14$. Here, we emphasize that $s$ is required to be greater than $1/4$  since $b > \frac12$.
	
	\item Case 3.2: $\la L_1\ra = \text{MAX}$.
	
	In this case, $\la L_1\ra \geq \la L_3\ra$, so
	$$
		\frac{1}{\la L_1 \ra^b \la L_3\ra^{1-b}} \leq \frac{1}{\la L_1\ra^{1-b} \la L_3\ra^b}.
	$$
	Recalling $ |\xi_1| \sim |\xi_3|$, so $ (\ref{bl1-bn-wl}) $ is reduce to
	$$
		\int \frac{ |\xi_1|^{1-s} \prod_{i=1}^3 |f_i|}{\la L_1\ra^{1-b} \la L_2\ra^b \la L_3\ra^b} \leq C \prod_{i=1}^3 \|f_i\|_{L^2}.
	$$
	Then similar to the derivation of (\ref{bl1-x3l-ds-sup3}), it then suffices to prove
	\be\label{bl1-bn-x3l-ds-sup1}
		\sup _{|\xi_1| \gg 1 , \tau_1 \in \m{R}} \frac{ |\xi_1|^{2 - 2s} }{\la L_1\ra^{2 - 2b} } \int \frac{d \xi_2}{\la L_2 + L_3 \ra^{2b} } \leq C.
	\ee
	
	For fixed $\xi_1$ and $\tau_1$, $L_2 + L_3$ can be viewed as a cubic function of $\xi_2$. More specifically,
	$L_2 + L_3 = -3 P_{\xi_1, \tau_1}(\xi_2)$, where
	\[
		P_{\xi_1, \tau_1}(\xi_2) = \xi_2^3 - \xi_1 \xi_2^2 - (\xi_1^2 + \b/3) \xi_2 - \frac{\xi_1^3 - \tau_1}{3}.
	\]
	So it follows from Lemma \ref{Lemma, int for cubic} that
	$$
		\int \frac{d \xi_2}{ \la L_2 + L_3\ra^{2b} } \leq C \big \la -3(\xi_1^2 + \b/3) - \xi_1^2\big\ra^{-\frac14} \sim |\xi_1|^{-\frac12}.
	$$
	
	In addition, thanks to the assumption $\la L_1\ra = \text{MAX}$, we know
	$$
		\la L_1 \ra \gs |H| \gs |\xi_3| \sim |\xi_1|.
	$$
	Consequently,
	$$\begin{aligned}
		\text { LHS of } (\ref{bl1-bn-x3l-ds-sup1}) 
		&\ls \sup _{ |\xi_1| \gg 1, \tau_1 \in \m{R} } \frac{ |\xi_1 |^{2 - 2s} }{ |\xi_1|^{2 - 2b} } \cdot |\xi_1|^{-\frac12} \\
		& = \sup _{ |\xi_1| \gg 1, \tau_1 \in \m{R} } |\xi_1|^{2b - 2s - \frac12} \ls 1,
	\end{aligned}$$
	where the last inequality is again due to the assumption that 
	$b \leq s + \frac14$.

	\item Case 3.3: $\la L_2 \ra = \text{MAX}$.
	
	It is easily seen that $(\xi_1, \tau_1)$ and $(\xi_2, \tau_2)$ are symmetric in (\ref{bl1-bn-wl}), so Case 3.3 can be handled in the same way as Case 3.2 by just switching $(\xi_1, \tau_1)$ and $(\xi_2, \tau_2)$.
\end{itemize}

Therefore, we finished the proof for the bilinear estimate (\ref{bilin-R1}) under condition (ii) in Proposition \ref{Prop, bilin-R}.

\subsection{Proof of (\ref{bilin-R2}) in Proposition \ref{Prop, bilin-R}}
\label{Subsec, bilin-2}

Although the justification of (\ref{bilin-R2}) is more challenging than that of (\ref{bilin-R1}) since $(\xi_1, \tau_1)$ and $(\xi_2, \tau_2)$ are not symmetric in (\ref{bilin-R2}), the essential idea and the general framework for these two bilinear estimates are still in the same spirit. So in the following, we will only sketch the key steps in the verification of (\ref{bilin-R2}).

Based on Lemma \ref{Lemma, bilin to weighted l2}, the bilinear estimate (\ref{bilin-R2}) is equivalent to the following estimate:
\be\label{bilin2-wl}
	\int_A \frac{ \xi_3 \la \xi_3 \ra^s \prod_{i=1}^3 f_i ( \xi_i, \tau_i ) }{ \la \xi_1\ra^s \la \xi_2 \ra^s \la L_1 \ra^b \la L_2 \ra^b \la L_3\ra^{1-b} } 
	\leq C \prod_{i=1}^3 \| f_i \|_{L^2}, \quad \forall f_1, f_2, f_3 \in \mathscr{S}(\m{R}^{2}),
\ee
where
\be\label{bilin2-L}
L_1 = \tau_1 - \xi_1^3, \quad L_2 = \tau_2 - \phi^{\a, \b} (\xi_2), \quad L_3 = \tau_3 - \phi^{\a, \b} (\xi_3).
\ee
This time, the resonance function $H$ becomes 

\be\label{res fun-2}
	H(\vec{\xi}, \vec{\tau}) = \sum_{i=1}^{3} L_i = 3\xi_1 \Big[ (\xi_2 - \xi_3)^2 - \frac{\b}{3} \Big], \quad \forall\, (\vec{\xi}, \vec{\tau}) \in A.
\ee

\subsubsection{Condition (i)}
\label{Subsubsec, bl2-bp}

We first deal with the case when $0< \b \leq 1$, $s \geq \frac12$ and $\frac12 < b \leq \frac34$. Similar to the proof in Section \ref{Subsubsec, bl1-bp}, we split the integral domain $A$ in (\ref{int domain}) into several pieces.

\begin{itemize}
	
	\item Case 1: $\|\xi_3| \ls 1$.
	
	Recalling Case 1 in Section \ref{Subsubsec, bl1-bp}, the essential step is the estimate of the integral 
	\[ 
		\int \frac{d \xi_2}{\la  L_1 + L_2 \ra^{2b} } 
	\] 
	in (\ref{bl1-x3s-sup3}).
	For the newly defined $L_1$ and $L_2$ in (\ref{bilin2-L}), we can regard $L_1 + L_2$ as a cubic function of $\xi_2$ for fixed $\xi_3$ and $\tau_3$. More precisely,
	$L_1 + L_2 = -3 P_{\xi_3, \tau_3}(\xi_2)$, where
	\be\label{L1aL2-2}
		P_{\xi_3, \tau_3}(\xi_2) = \xi_2^3 - \xi_3 \xi_2^2 - \Big(\xi_3^2 + \frac{\beta}{3} \Big) \xi_2 + \frac{\tau_3 - \xi_3^3}{3}.
	\ee
	Then it follows from Lemma \ref{Lemma, int for cubic} that
	\[
		\int \frac{d \xi_2}{\la L_1 + L_2\ra^{2b} } \ls 1.
	\]
	The rest argument can be carried out similarly as that in Case 1 in Section \ref{Subsubsec, bl1-bp}.
	
	Next, by noticing that the key ingredient in Case 2 in Section \ref{Subsubsec, bl1-bp} is to take advantage of the large size of $|P_{\xi_3, \tau_3}^{\prime}(\xi_2)|$. Now according to (\ref{L1aL2-2}),
	\be\label{bl2-L1aL2d}
	\begin{aligned}
		P_{\xi_3, \tau_3}^{\prime} (\xi_2) & = 3 \xi_2^2 - 2 \xi_3 \xi_2 - \Big(\xi_3^2 + \frac{\beta}{3} \Big) \\
		& = (\xi_2 - \xi_3)(3 \xi_2 + \xi_3) - \beta/3.
	\end{aligned}
	\ee
	This motivates us to introduce Case 2 below.

	\item Case 2: $ |\xi_3| \gg 1$, $|\xi_2 - \xi_3| \geq \frac{1}{10} |\xi_3|$ and $ |3 \xi_2 + \xi_3| \geq \frac{1}{10} |\xi_3|$.
	
	In this case, it is easily seen from (\ref{bl2-L1aL2d}) that
	$$
		\Big| P_{\xi_3, \tau_3}^{\prime} (\xi_2) \Big| \gs |\xi_3|^2 \gg 1.
	$$
	Then we can repeat the process in Case 2 in Section \ref{Subsubsec, bl1-bp} to finish this part.
	
	\item Case 3: $ |\xi_3| \gg 1$,  $|\xi_2 - \xi_3| < \frac{1}{10} |\xi_3|$ or $|3 \xi_2 + \xi_3| < \frac{1}{10} |\xi_3|$.
	
	In this case, we see that $ |\xi_1| \sim |\xi_2| \sim |\xi_3|$, so it suffices to prove
	\be\label{bl2-wl-as}
		\int \frac{ |\xi_3|^{1-s} \prod_{i=1}^3 f_i (\xi_i, \tau_i) }{\la L_1\ra^b \la L_2\ra^b \la L_3\ra^{1-b} } \leq C \prod_{i=1}^3 \|f_i\|_{L^2}, \quad \forall f_1,\, f_2,\, f_3 \in \mathscr{S}(\m{R}^2).
	\ee
	Unlike Case 3 in Section \ref{Subsubsec, bl1-bp}, the polynomial $L_1 + L_2 = P_{\xi_{3}, \tau_3}(\xi_2)$ now is of cubic degree as in (\ref{L1aL2-2}) rather than a quadratic degree as in (\ref{L1aL2}). Meanwhile, the estimate in Lemma \ref{Lemma, int for cubic} for cubic polynomials are not as strong as that for quadratic polynomials in Lemma \ref{Lemma, int for quad}, so the proof for Case 3 in Section \ref{Subsubsec, bl1-bp} does not work here.
	
	Fortunately, we can still adopt the approach in Case 3 in Section \ref{Subsubsec, bl1-bn} by taking advantage of the resonance function $H$ with suitable modifications. In Section \ref{Subsubsec, bl1-bn}, $\beta$ is negative so that $|H|$ automatically possesses a good lower bound $|\xi_3|$ in (\ref{bl1-bn-Hlb}), however, the resonance function $H$ in (\ref{res fun-2}) may be close to 0 since $\beta$ is positive in the current situation. Next, we will investigate when $H$ in (\ref{res fun-2}) is large.
	
	Since $\b$ is positive, there exists $\b_1>0$ such that 
	\be\label{beta1}
		\b_1^2 = \b / 3.
	\ee
	Then it follows from (\ref{res fun-2}) that
	\be\label{res-fun-2bp}
	\begin{aligned}
		H & = 3 \xi_1 \big[ (\xi_2 - \xi_3)^2 - \b_1^2 \big] \\
		& = 3 \xi_1 (\xi_2 - \xi_3 + \b_1)(\xi_2 - \xi_3 - \b_1).
	\end{aligned}
	\ee
	This suggests to split Case 3 into two subcases.
	
	\begin{itemize}
		\item Case 3.1: $| \xi_2 - \xi_3 + \b_1 | \leq \b_1 / 2$ or $ |\xi_2 - \xi_3 - \b_1 | \leq \b_1 / 2$.
		
		In this case, $H$ may be small according to (\ref{res-fun-2bp}), so we turn to look for a large lower bound for $ | P_{\xi_3, \tau_3}^{\prime}(\xi_2)|$ as in Case 2.
		In fact, we can infer from the assumption in Case 3.1 that
		$$
			|\xi_2 - \xi_3| \sim \b_1 \sim 1 \text { and } |3\xi_2 + \xi_3| \sim |\xi_3| \gg 1.
		$$
		Hence, it follows from (\ref{bl2-L1aL2d}) that
		$\big| P_{\xi_3, \tau_3}^{\prime} (\xi_2) \big| \gs |\xi_3|$.
		By taking advantage of this lower bound and the assumption $s \geq 1/2$, we can justify (2.37) in the similar manner as that in Case 2.
		
		\item Case 3.2: $| \xi_2 - \xi_3 + \b_1 | > \b_1 / 2$ and $|\xi_2 - \xi_3 - \b_1| > \b_1 / 2$. 
		
		In this case, it follows from (\ref{res-fun-2bp}) that
		\be\label{bl2-bp-Hlb}
			|H| \gs |\xi_1|.
		\ee
		Recalling the fact that $|\xi_1| \sim | \xi_2| \sim |\xi_3|$, so the lower bound (\ref{bl2-bp-Hlb}) is comparable to that in (\ref{bl1-bn-Hlb}) in Section \ref{Subsubsec, bl1-bn}. Then we can follow the argument in Case 3 in Section \ref{Subsubsec, bl1-bn} to justify (\ref{bl2-wl-as}) by considering further three subcases depending on which $\la L_i \ra$ is the maximum of $\big\{ \la L_1\ra, \, \la L_2\ra, \, \la L_3 \ra \big\}$.
	\end{itemize}
	We remark that it is required that $s \geq \frac12$ and $\frac12 < b \leq \frac34$ in Case 3.1 while it is only required $s > \frac14$ and $\frac12 < b \leq s + \frac14$ in Case 3.2.
\end{itemize}

Hence, we finished the proof for the bilinear estimate (\ref{bilin-R2}) under condition (i) in Proposition \ref{Prop, bilin-R}.

\subsubsection{Condition (ii)}
\label{Subsubsec, bl2-bn}

Now we continue to investigate the case when $-1 \leq \b < 0$, $s > \frac14$ and $\frac12 < b \leq s + \frac14$. The proof for this case can follow almost the same procedure as that in Section \ref{Subsubsec, bl2-bp} and actually be even simpler. 

More specifically, for Case 1 and Case 2 in Section \ref{Subsubsec, bl2-bp}, the arguments are exactly the same. For Case 3, the resonance function $H$ has better property. In fact, since $\b < 0$ in the current case, there exists some $\b_2 > 0$ such that 
\be\label{beta2}
	\b_2^2 = - \b / 3.
\ee
As a result, it follows from (\ref{res fun-2}) that 
\be\label{res-fun-2bn}
	H = 3 \xi_1 \big[ (\xi_2 - \xi_3)^2 + \b_2^2\big].
\ee
Thanks to (\ref{res-fun-2bn}) which ensures a lower bound $|\xi_1|$ for $|H|$, we can follow the argument in Case 3.2 in Section \ref{Subsubsec, bl2-bp} to justify (\ref{bl2-wl-as}) under the conditions $s > \frac14$ and $\frac12 < b \leq s + \frac14$.

Hence, we finished the proof for the bilinear estimate (\ref{bilin-R2}) under condition (ii) in Proposition \ref{Prop, bilin-R}.

\section{Proof of Theorem \ref{Thm, main_R}}
\label{Sec, pf-main}

For Theorem \ref{Thm, main_R}, we can take advantage of the bilinear estimates in Proposition \ref{Prop, bilin-R} and some standard argument (see e.g. \cite{KPV96, Oh09, YZ22a}) to justify the analytically LWP of (\ref{m-MB}) in $\mathscr{H}^{s}(\m{R})$ for $s \geq \frac12$ in part (a) and for $s > \frac14 $ in part (b). Next, we focus on the failure of $C^2$ LWP for part (a) when $s < \frac12$ and the failure of $C^3$ LWP for part (b) when $s < \frac14$. Based on \cite{Bou97, Oh09}, we first introduce the general framework.

We first fix a notation. Consider the Cauchy problem of the following linear KdV equation with $a, b \in \m{R}$ and $a \neq 0$.
\be\label{lKdV}
	\begin{cases}
		w_t + a w_{xxx} + b w_x = 0, \\ 
		w(x, 0) = w_0(x). 
	\end{cases}
\ee
For any $s \in \m{R}$ and $w_0 \in H^s(\m{R})$, it is well-known that (\ref{lKdV}) admits a unique solution $w \in C_b \big(\m{R}; H^{s}(\m{R})\big)$.

\begin{definition}\label{Def, KdVop}
Let $a \neq 0$, $b \in \m{R}$, $s \in \m{R}$ and $w_0 \in H^s(\m{R})$. We denote the unique solution $w$ of (\ref{lKdV}) to be 
$w(x, t) = \big[S_{a, b}(t) w_0\big](x)$ 
or simply 
$w(t) = S_{a, b}(t) w_0$.
In particular, when $a = 1$ and $b = 0$, we denote $S_{1,0}$ to be $S$ for short.
\end{definition}
Based on  (\ref{lKdV}) and Definition \ref{Def, KdVop}, it is straightforward to see that
\be\label{semiop}
	\mathscr{F}_{x}\big[ S_{a,b}(t) w_0 \big](\xi) 
	=  e^{i(a\xi^3 - b\xi )t} \wh{w}_0(\xi) 
	= e^{ i \phi^{a,b}(\xi) t} \wh{w}_0(\xi),
\ee
where $\mathcal{F}_{x}$ means the Fourier transform in the $x$ variable. 

Now we come hack to (\ref{m-MB}) with $\a = 4$ and $\b \in \m{R}$.
Recall that when we say (\ref{m-MB}) is locally well-posed with a $C^k(k \geq 1)$ solution map, it means that there exists $t>0$ such that the map from the initial data $(u_0, v_0) \in \mathscr{H}^s(\m{R})$ to the local solution $(u, v) \in C \big( [0, T]; \mathscr{H}^s(\m{R}) \big)$ is $C^k$.
According to \cite{Bou97, Oh09}, we take $(u_0, v_0) = (\delta \phi, \delta \psi)$ so that (\ref{m-MB}) becomes
\be\label{m-MB-ip}
	\left\{\begin{array}{l}
		u_t + u_{xxx} + v v_x = 0, \\
		v_t + \a v_{xxx} + \b v_x + (uv)_x = 0, \\
		\big( u(x, 0), v(x, 0) \big) = \big( \delta \phi(x), \delta \psi(x) \big),
	\end{array}\right.
\ee
where $\delta \geq 0$ and $(\phi, \psi) \in \mathscr{H}^s(\m{R})$.
We denote the solution of (\ref{m-MB-ip}) to be $\big( u(x, t, \delta), v(x, t, \delta) \big)$ or simply $\big( u(t, \delta), v(t, \delta) \big)$. Then it follows from the Duhamel's principle that
\be\label{Duh-soln}
	\left\{\begin{aligned}
		& u(t,\delta) = \delta S(t)\phi - \frac12 \int_{0}^{t} S(t-\tau) \p_{x}(v^2)(\tau) \,d\tau, \\
		& v(t,\delta) = \delta S_{\a,\b}(t)\psi - \int_{0}^{t} S_{\a,\b}(t-\tau) \p_{x}(uv)(\tau) \,d\tau.
	\end{aligned}\right.
\ee

When $\delta = 0$, the initial data in (\ref{m-MB-ip}) is 0 and the unique solution is also 0, which means $\big( u(t, 0), v(t, 0) \big) = (0,0)$. Then by taking derivative in $\delta$ at $0$, it follows from (\ref{Duh-soln}) that
\be\label{1st-deri}
	\left\{\begin{aligned}
		& \p_\delta u(t, 0) = S(t) \phi \triangleq \phi_1(t), \\
		& \p_\delta v(t, 0) = S_{\a, \b}(t) \psi \triangleq \psi_1(t).
	\end{aligned}\right.
\ee
By taking the second and third derivatives in $\delta$ at $0$, we have

\be\label{2nd-deri}
	\left\{\begin{aligned}
		&\p_\delta^2 u(t, 0) = - \int_0^t S(t - \tau) \p_x ( \psi_1^2 )(\tau) \,d \tau \triangleq \phi_2(t), \\ 
		&\p_\delta^2 v(t, 0) = -2 \int_0^t S_{\a, \b}(t - \tau) \p_x (\phi_1 \psi_1)(\tau) \,d \tau \triangleq \psi_2(t).
	\end{aligned}\right.
\ee
and
\be\label{3rd-deri}
	\left\{\begin{aligned}
		&\p_\delta^3 u(t, 0) = -3 \int_0^t S(t - \tau) \p_x ( \psi_1 \psi_2 )(\tau) \,d \tau \triangleq \phi_3(t), \\ 
		&\p_\delta^3 v(t, 0) = -3 \int_0^t S_{\a, \b}(t - \tau) \p_x (\phi_1 \psi_2 + \phi_2 \psi_1)(\tau) \,d \tau \triangleq \psi_3(t).
	\end{aligned}\right.
\ee

Note that if the solution map is $C^2$, then there exists $T_1 > 0$ such that
\be\label{2nd-derib}
	\sup _{0 \leq t \leq T_1} \big\| (\phi_2, \psi_2)(\cdot, t) \big\|_{ \mathscr{H}_x^s( \m{R} ) } \leq C \| (\phi, \psi) \|_{\mathscr{H}^s( \m{R} ) }^2.
\ee
Similarly, if the solution map is $C^3$, then there exists $T_2>0$ such that
\be\label{3rd-derib}
	\sup _{0 \leq t \leq T_2}\big\| ( \phi_3, \psi_3 )(\cdot, t) \big\|_{\mathscr{H}_x^s( \m{R} ) } \leq C \|(\phi, \psi)\|_{\mathscr{H}^s( \m{R} ) }^3.
\ee


\subsection{Part (a)}
\label{Subsec, pf of Thm-a}

First, we infer from (\ref{1st-deri}) and (\ref{semiop}) that 
\be\label{1stF}
	\F_x \phi_1(\xi,t) = e^{i\xi^3 t} \wh{\phi}(\xi), \quad \F_x \psi_1(\xi,t) = e^{i \phi^{\a,\b}(\xi) t} \wh{\psi}(\xi),
\ee
where $\phi$ and $\psi$ will be determined later. Next, according to (\ref{2nd-deri}) and (\ref{semiop}), we have 
\[
	\F_{x}\psi_2(\xi,t) = -2 \int_{0}^{t} e^{i \phi^{\a,\b}(\xi)(t-\tau)} \xi \int_{\m{R}} \F_{x}\phi_1(\xi_1, \tau) \F_{x}\psi_1(\xi - \xi_1, \tau) \,d\xi_1 \,d\tau.
\]
Plugging (\ref{1stF}) into the above equation yields 
\[
	\begin{aligned}
		\F_{x} \psi_2(\xi, t) & = -2 \xi e^{i \phi^{\a, \b}(\xi) t} \int_{0}^{t} e^{-i \phi^{\a,\b}(\xi)\tau} \int_{\m{R}} e^{i\xi_1^3 \tau} \wh{\phi}(\xi_1) e^{i \phi^{\a,\b}(\xi - \xi_1) \tau} \wh{\psi}(\xi - \xi_1) \,d\xi_1 \, d\tau \\
		& = -2 \xi e^{i \phi^{\a, \b}(\xi) t} \int_{\m{R}} \wh{\phi}(\xi_1) \wh{\psi}(\xi - \xi_1) \bigg(\int_{0}^{t} e^{i G_0(\xi_1, \xi - \xi_1, -\xi) \tau} \,d\tau \bigg) \,d\xi_1,
	\end{aligned}
\]
where 
\be\label{G0}
	G_0(\eta_1, \eta_2, \eta_3) \triangleq \eta_1^3 + \phi^{\a,\b}(\eta_2) + \phi^{\a,\b}(\eta_3), \quad \forall\, \sum_{i=1}^{3} \eta_i = 0.
\ee
After integrating with respect to $\tau$,
\be\label{2ndF}
	\F_{x} \psi_2(\xi, t) = 2 i \xi e^{i \phi^{\a, \b}(\xi) t} \int_{\m{R}} \wh{\phi}(\xi_1) \wh{\psi}(\xi - \xi_1) \frac{e^{i G_0(\xi_1, \xi - \xi_1, -\xi) t} - 1}{G_0(\xi_1, \xi - \xi_1, -\xi)} \,d\xi_1.
\ee

Let $N$ be any large positive integer such that $N \gg 1 + \b_1$, where $\b_1 = \sqrt{\b / 3}$ as defined in (\ref{beta1}). Then we choose $\phi$ and $\psi$ such that 
\be\label{initF}
\left\{\begin{aligned} 
	& \wh{\phi}(\xi) = \g^{-\frac12} N^{-s} \, \m{I}_{[0,\g]} (\xi - 2N - \b_1), \\
	& \wh{\psi}(\xi) = \g^{-\frac12} N^{-s} \, \m{I}_{[\g, 2\g]} (\xi + N),
\end{aligned}\right.
\ee
where $\g = N^{-1}$ and $\m{I}$ represents the indicator function. Then it is easily seen that 
\be\label{init-n}
\| \phi \|_{H^s(\m{R})} \sim \| \psi \|_{H^s(\m{R})} \sim 1.	
\ee
For fixed $\xi$, we denote
\be\label{intd}
	A_{\xi} = \big\{ \xi_1\in\m{R}: 0\leq \xi_1 - 2N -\b_1 \leq \g,\, \g \leq \xi - \xi_1 + N \leq 2\g \big\}.
\ee
Meanwhile, by direct computation and using the fact that $\a = 4$ and $\b_1^2 = \b / 3$, we find
\be\label{res-f1}
	G_0(\xi_1, \xi - \xi_1, -\xi) = -3 \xi_1 \big[ (2\xi - \xi_1)^2 - \b_1^2 \big].
\ee
Then it follows from (\ref{2ndF}) that
\be\label{2ndFs}
	\F_{x} \psi_2(\xi,t) = 2 i \g^{-1} N^{-2s} \xi e^{i \phi^{\a,\b}(\xi)t} \int_{A_\xi} \frac{e^{ -3 i \xi_1 [ (2\xi - \xi_1)^2 - \b_1^2 ]t } -1 }{-3 \xi_1 [ (2\xi - \xi_1)^2 - \b_1^2 ]} \, d\xi_1.
\ee

Next, we focus on estimating the size of $|\F_{x} \psi_2(\xi,t)|$ when $0 < t \ll 1$ and $\xi \in E_1$, where 
\be\label{effrange}
	E_1 \triangleq \Big[ N + \b_1 + \frac74 \g, \, N + \b_1 + \frac94 \g \Big].
\ee
When $\xi \in E_1$ and $\xi_1 \in A_\xi$, it follows from (\ref{intd}) and (\ref{effrange}) that 
\[ 
	\xi_1 \sim N \quad \text{and} \quad 2\g \leq 2\xi - \xi_1 - \b_1 \leq 5\g, 
\]
which implies that 
\be\label{res-size1}  
	\xi_1 \big[ (2\xi - \xi_1)^2 - \b_1^2 \big] \sim N\g \sim 1.
\ee
As a result, there exists some absolute constant $t_1$ (independent of $N$) such that for any $0 < t \leq t_1$,
\be\label{large-im}
	\text{Im}\bigg( \frac{e^{ -3 i \xi_1 [ (2\xi - \xi_1)^2 - \b_1^2 ]t } -1 }{-3 \xi_1 [ (2\xi - \xi_1)^2 - \b_1^2 ]} \bigg) 
	= \frac{\sin\big( 3 \xi_1 [ (2\xi - \xi_1)^2 - \b_1^2 ] t \big) }
	{3 \xi_1 [ (2\xi - \xi_1)^2 - \b_1^2 ]} \geq \frac{t}{2}.
\ee
Taking advantage of (\ref{large-im}), it follows from (\ref{2ndFs}) that for any $\xi \in E_1$, 
\[  
	\big| \F_{x}\psi_2(\xi,t) \big| \geq \g^{-1} N^{-2s} \xi t |A_\xi|,
\]
where $|A_\xi|$ refers to the measure of the set $A_\xi$. When $\xi\in E_1$, by straightforward check, we infer from (\ref{intd}) that 
\[  
	\Big[ 2N + \b_1 + \frac14 \g, \, 2N + \b_1 + \frac34 \g  \Big] \subset A_\xi \subset [ 2N + \b_1, \, 2N + \b_1 + \g ].
\]
So $\g / 2 \leq |A_\xi | \leq \g$ and 
\[  
	|\F_x \psi_2(\xi,t)| \geq \g^{-1} N^{-2s} N t \frac{\g}{2} = \frac{t}{2} N^{1-2s}.
\]
Therefore, 
\be\label{psi2norm}
	\begin{split}
		\| \psi_2(\cdot, t) \|_{H^s}^2 &\geq \Big\| \la \xi \ra^s \F_x \psi_2(\xi,t) \m{I}_{E_1}(\xi) \Big\|_{L^2_\xi}^2 \\
		& \geq \int_{E_1} \la \xi \ra^{2s} \Big( \frac{t}{2} N^{1-2s} \Big)^2 \,d\xi \geq \frac{t^2}{8} N^{1-2s}.
	\end{split}
\ee

If the solution map is $C^2$, then it follows from (\ref{2nd-derib}), (\ref{psi2norm}) and (\ref{init-n}) that 
\[  
	\frac{t^2}{8} N^{1 - 2s} \leq C,
\]
where $0 < t \leq \min\{ T_1, t_1 \}$ and $C$ is some constant independent of $N$. Hence, by fixing $t = \min\{ T_1, t_1\}$ and sending $N \to \infty$, we conclude that $s\geq \frac12$.

\subsection{Part (b)}
\label{Subsec, pf of Thm-b}

Firstly, we choose $\phi = 0$ in (\ref{m-MB-ip}) such that $\phi_1 = \psi_2 = \phi_3 = 0$ according to (\ref{1st-deri})--(\ref{3rd-deri}). Meanwhile, similar to (\ref{1stF}) and (\ref{2ndF}) in Section \ref{Subsec, pf of Thm-a}, it follows from (\ref{1st-deri}) and (\ref{2nd-deri}) that 
\be\label{1stF-b}
	\F_x \psi_1(\xi, t) = e^{i \phi^{\a,\b}(\xi) t} \wh{\psi}(\xi),
\ee
and 
\be\label{2ndF-b}
	\F_x \phi_2(\xi, t) = i \xi e^{i \xi^3 t} \int_{\m{R}} \wh{\psi}(\xi_1) \wh{\psi}(\xi - \xi_1) \frac{e^{i G_1(\xi_1, \xi - \xi_1, -\xi)t} -1 }{G_1(\xi_1, \xi - \xi_1, -\xi)} \,d\xi_1,
\ee
where 
\be\label{G1}
	G_1(\eta_1, \eta_2, \eta_3) \triangleq \phi^{\a,\b}(\eta_1) + \phi^{\a,\b}(\eta_2) + \eta_3^3, \quad \forall\, \sum_{i=1}^{3} \eta_i = 0.
\ee

Next, we will focus on the computation of $\F_x \psi_3(\xi, t)$ when $0 < t \ll 1$. Since $\phi_1 = \psi_2 = 0$, we infer from (\ref{3rd-deri}) that 
\[  
	\psi_3(t) = -3 \int_{0}^{t} S_{\a,\b}(t-\tau) \p_x(\phi_2 \psi_1)(\tau) \,d\tau.
\]
Based on (\ref{semiop}), we find 
\[  
	\F_x \psi_3(\xi,t) = -3 \int_{0}^{t} e^{i \phi^{\a,\b}(\xi)(t - \tau)} \xi \int_{\m{R}} \F_x \psi_1 (\xi_1, \tau) \F_x \phi_2(\xi - \xi_1, \tau) \,d\xi_1 \,d\tau.
\]
Plugging (\ref{1stF-b}) and (\ref{2ndF-b}) into the above formula yields 
\[  
	\begin{split}
		\F_{x} \psi_3(\xi, t) = & -3 \int_0^t e^{i \phi^{\a,\b}(\xi) (t - \tau)} \xi \int_{\m{R}} e^{i \phi^{\a,\b}(\xi_1) \tau} \wh{\psi}(\xi_1) \bigg[ i (\xi - \xi_1) e^{i (\xi - \xi_1)^3 \tau} \\
		& \int_{\m{R}} \wh{\psi}(\xi_2) \wh{\psi}(\xi - \xi_1 -\xi_2) \frac{e^{i G_1(\xi_2, \xi - \xi_1 - \xi_2, \xi_1 - \xi)\tau} - 1}{G_1(\xi_2, \xi - \xi_1 - \xi_2, \xi_1 - \xi)} \,d\xi_2 \bigg] \,d\xi_1 \,d\tau.
	\end{split}
\]
Rearranging terms leads to
\be\label{3rdF}
	\begin{split}
		\F_x \psi_3(\xi, t) = & -3 i \xi e^{ i \phi^{\a,\b}(\xi) t } \iint_{\m{R}^2} \frac{(\xi - \xi_1) \wh{\psi}(\xi_1) \wh{\psi}(\xi_2) \wh{\psi}(\xi - \xi_1 -\xi_2) }{G_1(\xi_2, \xi - \xi_1 - \xi_2, \xi_1 - \xi)} \\
		& \int_{0}^{t} e^{i [ \phi^{\a,\b}(\xi_1) + (\xi - \xi_1)^3 - \phi^{\a,\b}(\xi) ]\tau} \Big( e^{i G_1(\xi_2, \xi - \xi_1 - \xi_2, \xi_1 - \xi) \tau } - 1 \Big) \,d\tau \,d\xi_2 \,d\xi_1.
	\end{split}
\ee
Noticing 
\[  
	\phi^{\a,\b}(\xi_1) + (\xi - \xi_1)^3 - \phi^{\a,\b}(\xi) = G_1(\xi_1, -\xi, \xi - \xi_1)
\]
and 
\[
	\begin{split}
		& G_1(\xi_1, -\xi, \xi - \xi_1) + G_1(\xi_2, \xi - \xi_1 - \xi_2, \xi_1 - \xi) \\
		= \quad & \phi^{\a, \b}(\xi_1) + \phi^{\a, \b}(-\xi) + \phi^{\a, \b}(\xi_2) + \phi^{\a, \b}(\xi - \xi_1 - \xi_2)
		\\
		= \quad & G_2(\xi_1, \xi_2, \xi - \xi_1 -\xi_2, -\xi),
	\end{split}
\]
where 
\be\label{G2}
	G_2(\eta_1, \eta_2, \eta_3, \eta_4) \triangleq \phi^{\a, \b}(\eta_1) + \phi^{\a, \b}(\eta_2) + \phi^{\a, \b}(\eta_3) + \phi^{\a, \b}(\eta_4), \quad \forall\, \sum_{i=1}^{3} \eta_i = 0,
\ee
it then follows from (\ref{3rdF}) that 
\be\label{3rdFdcom}
	\F_x \psi_3(\xi, t) = -3 i \big[ I_1(\xi,t) - I_2(\xi,t)  \big],
\ee
where
\be\label{I1}
	\begin{split}
		I_1(\xi,t) = \, & \xi e^{ i \phi^{\a,\b}(\xi) t } \iint_{\m{R}^2} \frac{(\xi - \xi_1) \wh{\psi}(\xi_1) \wh{\psi}(\xi_2) \wh{\psi}(\xi - \xi_1 -\xi_2) }{G_1(\xi_2, \xi - \xi_1 - \xi_2, \xi_1 - \xi)} \\
		& \quad \frac{e^{i G_2(\xi_1, \xi_2, \xi - \xi_1 -\xi_2, -\xi) t} - 1}{G_2(\xi_1, \xi_2, \xi - \xi_1 -\xi_2, -\xi)} \,d\xi_2 \,d\xi_1
	\end{split}
\ee
and 
\be\label{I2}
	\begin{split}
		I_2(\xi,t) = \, & \xi e^{ i \phi^{\a,\b}(\xi) t } \iint_{\m{R}^2} \frac{(\xi - \xi_1) \wh{\psi}(\xi_1) \wh{\psi}(\xi_2) \wh{\psi}(\xi - \xi_1 -\xi_2) }{G_1(\xi_2, \xi - \xi_1 - \xi_2, \xi_1 - \xi)} \\
		& \quad \frac{e^{i G_1(\xi_1, -\xi, \xi - \xi_1) t} - 1 }{G_1(\xi_1, -\xi, \xi - \xi_1)} \,d\xi_2 \,d\xi_1.
	\end{split}
\ee
  
Now we are ready to choose a suitable $\psi$ such that $I_1$ is the dominating term over $I_2$. Firstly, based on (\ref{G1}) and (\ref{G2}) with $\a = 4$, we have 
\be\label{G12}
	\left\{\begin{aligned}
		& G_1(\eta_1, \eta_2, \eta_3) = -3 \eta_3 \big[ (\eta_1 - \eta_2)^2 - \b / 3 \big],  \quad \forall\, \sum_{i=1}^{3} \eta_i = 0,  \\
		& G_2(\eta_1, \eta_2, \eta_3, \eta_4) = -12 (\eta_1 + \eta_2)(\eta_1 + \eta_3)(\eta_2 + \eta_3), \quad \forall\, \sum_{i=1}^{4} \eta_i = 0.
	\end{aligned}\right.
\ee
Since $\b < 0$, there exists $\b_2 > 0$ such that $\b_2^2 = -\b / 3$. Consequently, 
\be\label{Gform}
	\left\{\begin{aligned}
		& G_1(\xi_2, \xi - \xi_1 -\xi_2, \xi_1 - \xi) = -3 (\xi_1 - \xi) \big[ (\xi_1 + 2\xi_2 - \xi)^2 + \b_2^2 \big],  \\
		& G_1(\xi_1, -\xi, \xi - \xi_1) = -3 (\xi - \xi_1) \big[ (\xi_1 + \xi)^2 + \b_2^2 \big], \\
		& G_2(\xi_1, \xi_2, \xi - \xi_1 -\xi_2, -\xi) = -12 (\xi_1 + \xi_2)(\xi - \xi_2)(\xi - \xi_1).
	\end{aligned}\right.
\ee

Let $N$ be any positive integer such that $N \gg 1 + |\b|$. Then (\ref{Gform}) inspires us to choose $\psi$ such that 
\be\label{r-psi}
	\wh{\psi}(\xi) = \g_2^{-1/2} N^{-s} \, \m{I}_{B}(\xi),
\ee
where $\g_2 = N^{-\frac12}$ and $B = B_1 \cup B_2$ with 
\be\label{B12}
	B_1 = [N, N + 4\g_2], \quad B_2 = [ -N - 9\g_2, -N - 5\g_2 ].
\ee
Then it is readily seen that 
\be\label{init-n2}
	\| \psi \|_{H^{s}(\m{R})} \sim 1.
\ee

Next, we will estimate the size of $\F_x \psi_3(\xi, t)$ when $\xi \in E_2$ and $0 < t \ll 1$, where 
\be\label{effrange2}
	E_2 \triangleq [ -N - 12 \g_2, \, -N - 11 \g_2 ].
\ee
Since the support $B$ of $\psi$ consists of two parts $B_1$ and $B_2$, we can decompose $I_1$ and $I_2$ in (\ref{I1}) and (\ref{I2}) into four parts respectively. More specifically, we write 
\be\label{I1decom}
	I_1 = I_{11} + I_{12} + I_{13} + I_{14},
\ee
where $I_{11}$, $I_{12}$, $I_{13}$ and $I_{14}$ denote the contribution to $I_1$ on the integral domains $B_1 \times B_1$, $B_1 \times B_2$, $B_2 \times B_1$ and $B_2 \times B_2$ respectively. Analogously, we write 
\be\label{I2decom}
	I_2 = I_{21} + I_{22} + I_{23} + I_{24}.
\ee

\begin{itemize}
	\item Case 1: $\xi_1, \xi_2 \in B_1$, where $B_1$ is as defined in (\ref{B12}).
	
	In this case, since we only consider $\xi \in E_2$ as defined in (\ref{effrange2}), then it is straightforward  to check that $ \xi - \xi_1 - \xi_2 \leq -3N - 11 \g_2$, which implies that $\wh{\psi}(\xi - \xi_1 - \xi_2) = 0$. Hence, 
	\be\label{est-d1}
		| I_{11}(\xi,t) | + | I_{21}(\xi,t) | = 0, \quad \forall \, \xi\in E_2, \, t>0.
	\ee
	
		
	\item Case 2: $\xi_1 \in B_1$ and $\xi_2 \in B_2$, where $B_1$ and $B_2$ are as defined in (\ref{B12}).
	
	Again, we restrict the consideration of $\xi$ in $E_2$ in this case, it then follows from (\ref{Gform}) that
	\[
	\begin{split}
		& \big| G_1(\xi_2, \xi - \xi_1 - \xi_2, \xi_1 - \xi) \big| \sim N,\\
		& \big| G_1(\xi_1, -\xi, \xi - \xi_1) \big| \sim N, \\
		& \big| G_2(\xi_1, \xi_2, \xi - \xi_1 - \xi_2, -\xi) \big| \sim \g_2^2 N \sim 1.
	\end{split}	
	\]
	As a result, we infer from (\ref{I2}) that 
	\[  
		| I_{22}(\xi,t) | \leq C N \iint _{B_1 \times B_2} \frac{N \g_2^{-3/2} N^{-3s} }{N} \frac{2}{N} \,d\xi_2 \,d\xi_1 = C N^{-\frac14 - 3s}.
	\]
	Next, instead of providing an upper bound, we will derive a large lower bound for $|I_{12}(\xi, t) |$ which eventually becomes the dominating term. 
	
	Firstly, since $\text{Im}(e^{i \th} - 1 ) = \sin(\th)$ for any $\th \in \m{R}$, it then follows from (\ref{I1}) that 
	\be\label{I12}
		\begin{split}
			| I_{12}(\xi,t) | \geq |\xi| \bigg| \iint_{B_1 \times B_2} 
			& \frac{(\xi - \xi_1) \wh{\psi}(\xi_1) \wh{\psi}(\xi_2) \wh{\psi}(\xi - \xi_1 -\xi_2) }{G_1(\xi_2, \xi - \xi_1 - \xi_2, \xi_1 - \xi)} \\
			& \frac{\sin\big( G_2(\xi_1, \xi_2, \xi - \xi_1 - \xi_2, -\xi) t \big)}{G_2(\xi_1, \xi_2, \xi - \xi_1 - \xi_2, -\xi)} \,d\xi_2 \,d\xi_1 \bigg|.
		\end{split}
	\ee
	Recalling the property that $| G_2(\xi_1, \xi_2, \xi - \xi_1 - \xi_2, -\xi) | \sim 1 $, so there exists some absolute constant $t_2$ (independent of $N$) such that 
	\[  
		\frac{\sin\big( G_2(\xi_1, \xi_2, \xi - \xi_1 - \xi_2, -\xi) t \big)}{G_2(\xi_1, \xi_2, \xi - \xi_1 - \xi_2, -\xi)} \geq \frac{t}{2}, \quad \forall \, 0 < t \leq t_2.
	\]
	Moreover, when $\xi_1 \in B_1$ and $\xi \in E_2$, one can also check that \[
		\xi - \xi_1 < 0 \quad \text{and} \quad G_1(\xi_2, \xi - \xi_1 - \xi_2, \xi_1 - \xi) < 0.
	\] 
	This guarantees that the integrand  in (\ref{I12}) is always nonnegative, which allows one to move the absolute value sign into the integral. That is,
	\[  
		\begin{split}
			\text{RHS of (\ref{I12})} = |\xi| \iint_{B_1 \times B_2} 
			& \bigg| \frac{(\xi - \xi_1) \wh{\psi}(\xi_1) \wh{\psi}(\xi_2) \wh{\psi}(\xi - \xi_1 -\xi_2) }{G_1(\xi_2, \xi - \xi_1 - \xi_2, \xi_1 - \xi)} \\
			& \frac{\sin\big( G_2(\xi_1, \xi_2, \xi - \xi_1 - \xi_2, -\xi) t \big)}{G_2(\xi_1, \xi_2, \xi - \xi_1 - \xi_2, -\xi)} \bigg| \,d\xi_2 \,d\xi_1.
		\end{split}
	\]
	Therefore,
	\be\label{I12b1}
		\begin{split}
			| I_{12}(\xi, t) | &\gs N \iint_{B_1 \times B_2}  \frac{N \wh{\psi}(\xi_1) \wh{\psi}(\xi_2) \wh{\psi}(\xi - \xi_1 -\xi_2) }{N} \frac{t}{2} \,d\xi_2 \,d\xi_1 \\
			& \gs N t  \iint_{B_1 \times B_2}  \wh{\psi}(\xi_1) \wh{\psi}(\xi_2) \wh{\psi}(\xi - \xi_1 -\xi_2) \,d\xi_2 \,d\xi_1.
		\end{split}
	\ee
	Again, since $\xi \in E_2$, one can directly check that 
	$\wh{\psi}(\xi_1) = \wh{\psi}(\xi_2) = \wh{\psi}(\xi - \xi_1 - \xi_2) = 1$ 
	as long as
	\[  
		\xi_1 \in \Big[ N + \frac32 \g_2, \, N + \frac52 \g_2 \Big] \quad \text{and} \quad 
		\xi_2 \in \Big[ -N - \frac{15}{2} \g_2, \, -N - \frac{13}{2} \g_2 \Big].
	\]
	Thus, we infer from (\ref{I12b1}) that 
	\be\label{I12b2}
		| I_{12}(\xi, t) | \gs N t \int_{N + \frac32 \g_2}^{N + \frac52 \g_2} \int_{-N - \frac{15}{2} \g_2}^{-N - \frac{13}{2} \g_2} \g_2^{-\frac32} N^{-3s} \,d\xi_2 \,d\xi_1 \gs N^{\frac34 - 3s} t.	
	\ee
	As a consequence of (\ref{I12b1}) and (\ref{I12b2}), when $N$ is large,
	\be\label{est-d2}
		| I_{12}(\xi,t) | \geq C N^{\frac34 - 3s} t + |I_{22}(\xi, t)|, \quad \forall\, \xi\in E_2, \, 0 < t \leq t_2.
	\ee
		
	\item Case 3: $\xi_1 \in B_2$ and $\xi_2 \in B_1$. 
	
	In this case, it again follows from the constraint $\xi\in E_2$ and (\ref{Gform}) that 
	\[
	\begin{split}
		& \big| G_1(\xi_2, \xi - \xi_1 - \xi_2, \xi_1 - \xi) \big| \sim \g_2 N^2 \sim N^{\frac32},\\
		& \big| G_1(\xi_1, -\xi, \xi - \xi_1) \big| \sim \g_2 N^2 \sim N^{\frac32}, \\
		& \big| G_2(\xi_1, \xi_2, \xi - \xi_1 - \xi_2, -\xi) \big| \sim \g_2^2 N \sim 1.
	\end{split}	
	\]
	Then we use the similar estimate method as that in Case 1 to find
	\[  
		| I_{13}(\xi,t) | \leq C N^{-\frac54 - 3s} \quad \text{and} \quad |I_{23}(\xi,t)| \leq C N^{-\frac{11}{4} - 3s}.
	\]
	Consequently, when $N$ is large, 
	\be\label{est-d3}
		| I_{13}(\xi,t) | + | I_{23}(\xi,t) | \leq C N^{-\frac54 - 3s}, \quad \forall\, \xi\in E_2, \, t>0.
	\ee
	
	\item Case 4: $\xi_1, \xi_2 \in B_2$.
	
	This case is similar to Case 3, so we omit the details and only carry out the final result. That is, when $N$ is large,
	\be\label{est-d4}
		| I_{14}(\xi,t) | + | I_{24}(\xi,t) | \leq C N^{-\frac54 - 3s}, \quad \forall\, \xi\in E_2, \, t>0.
	\ee	
\end{itemize}

According to (\ref{est-d1}), (\ref{est-d2}), (\ref{est-d3}) and (\ref{est-d4}), it then follows from (\ref{3rdFdcom}) that 
\[  
	|\F_x \psi_3 (\xi,t)| = 3 \big| I_1(\xi,t) - I_2(\xi,t) \big| \geq C N^{\frac34 - 3s} t,
\]
provided that $\xi \in E_2$, $0 < t \leq t_2$ and $N$ is sufficiently large.
Hence, for any $0 < t \leq t_2$,
\be\label{psi3n}
	\begin{split}
		\| \psi_3(\cdot, t) \|_{H^s(\m{R})}^2 & \geq \Big\| \la \xi \ra^{s} \F_x \psi_3(\xi, t) \m{I}_{E_2}(\xi) \Big\|_{L^2(\m{R})}^2 \\
		& = \int_{E_2} \la \xi \ra^{2s} \big| \F_x \psi_3(\xi,t) \big|^2 \,d\xi \\
		& \gs \int_{E_2} N^{2s} N^{\frac32 - 6s} t^2 \,d\xi \sim N^{1 - 4s} t^2.
	\end{split}
\ee
If the solution map is $C^3$, then we infer from (\ref{3rd-derib}), (\ref{init-n2}) and (\ref{psi3n}) that 
\[  
	N^{1-4s} t^2 \leq C,
\]
where $0 < t \leq \min\{T_2, t_2\}$ and $C$ is some constant independent of $N$. Finally, by fixing $t = \min\{T_2, t_2\}$ and sending $N \to \infty$, we conclude that $s \geq \frac14$.

\section{Proof of Theorem \ref{Thm, auxi_R}}
\label{Sec, pf-auxi}

\subsection{Part (a)}
\label{Subsec, pf of Thm-auxi-a}

The general framework of this part is similar to that in Section \ref{Subsec, pf of Thm-a}. Recalling (\ref{2ndF}) in which we have derived that
\be\label{r-Fpsi2}
	\F_{x} \psi_2(\xi, t) = 2 i \xi e^{i \phi^{\a, \b}(\xi) t} \int_{\m{R}} \wh{\phi}(\xi_1) \wh{\psi}(\xi - \xi_1) \frac{e^{i G_0(\xi_1, \xi - \xi_1, -\xi) t} - 1}{G_0(\xi_1, \xi - \xi_1, -\xi)} \,d\xi_1,
\ee
where $G_0$ is as defined in (\ref{G0}). By replacing $\eta_2$ with $- (\eta_1 + \eta_3)$ in (\ref{G0}), we have 
\be\label{r-G}
	G_0(\eta_1, \eta_2, \eta_3) = -3\a \eta_1^3 f\Big(\frac{\eta_3}{\eta_1}\Big) + \b \eta_1, \quad \forall\, \sum_{i=1}^{3} \eta_i = 0.
\ee
where 
\be\label{r-qf}
	f(x) \triangleq x^2 + x + \frac{\a-1}{3\a}.
\ee
When $\a \in (0,4)\setminus\{1\}$, $f$ has two non-zero roots, $C_{1\a}$ and $C_{2\a}$:
\be\label{roots}
	C_{1\a} = - \frac12 - \frac12 \sqrt{ \frac{4-\a}{3\a} }, \qquad C_{2\a} = - \frac12 + \frac12 \sqrt{ \frac{4-\a}{3\a} }.
\ee
So $f$ and $G_0$ can be rewritten as $f(x) = ( x - C_{1\a} ) ( x - C_{2\a} ) $ and 
\be\label{r-G-comp}
	G_0(\eta_1, \eta_2, \eta_3) = \eta_1 \Big[ -3\a (\eta_3 - C_{1\a}\eta_1) (\eta_3 - C_{2\a} \eta_1 ) + \b \Big].
\ee
Let $N$ be any large positive integer and we choose 
\[
	\left\{\begin{array}{l}
		\wh{\phi}(\xi) = \g^{-\frac12} N^{-s} \m{I}_{0,\g}(\xi - N), \\
		\wh{\psi}(\xi) = \g^{-\frac12} N^{-s} \m{I}_{[-\g, \g]} (\xi + (1 + C_{1\a})N - \a N^{-1} ),
	\end{array}\right.
\]
where $\g = N^{-2}$ and $\lam$ is a real constant which will be determined later. It is easily seen that $\| \phi \|_{ H^s( \m{R}^n ) } \sim \| \psi \|_{ H^s( \m{R}^n ) } \sim 1$. For fixed $\xi \in \m{R}^n$, we denote 
\[  
	A_{\xi} = \big\{ \xi_1\in\m{R}: \quad 0\leq \xi_1 - N \leq \g, \quad -\g \leq \xi - \xi_1 + (1+C_{1\a}) N - \lam N^{-1} \leq \g \big\}.
\]
Then we infer from (\ref{r-Fpsi2}) that 
\be\label{r-Fpsi2-m}
	| \F_{x} \psi_2(\xi, t) | = 2 \g^{-1} N^{-2s} |\xi| \bigg| \int_{ A_{\xi} } \frac{e^{i G_0(\xi_1, \xi - \xi_1, -\xi) t} - 1}{G_0(\xi_1, \xi - \xi_1, -\xi)} \,d\xi_1 \bigg|.
\ee
When restricting $\xi$ on the interval $\big[ -C_{1\a}N + \lam N^{-1}, \, -C_{1\a}N + \lam N^{-1} + \g \big]$, we can easily verify that 
\[
	A_{\xi} = [N, N+\g].
\] 
Meanwhile, it follows from (\ref{r-G-comp}) that 
\[  
	G_0(\xi_1, \xi - \xi_1, -\xi) = \xi_1 \big[ -3\a (\xi + C_{1\a} \xi_1) (\xi + C_{2\a}\xi_1) + \b \big].
\]
When $N \leq \xi_1 \leq N + \g$, we find 
\[  
	\xi + C_{1\a} \xi_1 = \lam N^{-1} + O(N^{-2}), \qquad \xi + C_{2\a} \xi_1 = (C_{2\a} - C_{1\a}) N + O(N^{-1}).
\]
Therefore, 
\[  
	G_0(\xi_1, \xi - \xi_1, -\xi) = \xi_1 \big[ -3\a \lam (C_{2\a} - C_{1\a}) + \b + O(N^{-1}) \big].
\]
Now we choose $\lam$ such that $-3\a \lam (C_{2\a} - C_{1\a}) + \b = 0$, that is, 
\[  
	\lam = \frac{\b}{ 3\a ( C_{2\a} - C_{1\a} ) } = \frac{\b}{\sqrt{3\a (4-\a)}}.
\]
Hence, 
\[  
	| G_0(\xi_1, \xi - \xi_1, -\xi) | \sim O(N^{-1} |\xi_1|) \sim O(1).
\]
The rest argument is similar to that in Section \ref{Subsec, pf of Thm-a} (from (\ref{res-size1}) to the end) and we finally conclude that $s\geq 0$ if the solution map is $C^2$.

\subsection{Part (b)}h
\label{Subsec, pf of Thm-auxi-b}
The general framework of this part is similar to that in Section \ref{Subsec, pf of Thm-b}, so we will only highlight the key step without carrying out all details. We first choose $\phi = 0$ in (\ref{m-MB-ip}) and recall $\F_x \psi_3$, $I_1$ and $I_2$ in (\ref{3rdFdcom})-(\ref{I2}), the crucial parts are functions $G_1$ and $G_2$ which were introduced earlier in (\ref{G1}) and (\ref{G2}). We are going to analyze these two functions for general $\a$ and $\b$.

Firstly, by direct computation, 
\be\label{gG2}
	G_2(\eta_1, \eta_2, \eta_3, \eta_4) = -3\a (\eta_1 + \eta_2)(\eta_1 + \eta_3)(\eta_2 + \eta_3),
\qquad \forall \sum_{i=1}^{4} \eta_i = 0.
\ee
This means that $G_2$ is independent of $\b$ and the effect of $\a$ is only a constant multiple. 
Secondly, according to the definitions for $G_0$ and $G_1$, we know $G_1(\eta_1, \eta_2, \eta_3) = G_0(\eta_3, \eta_2, \eta_1)$. So it follows from (\ref{G0}) that 
\be\label{gG1}
G_1(\eta_1, \eta_2, \eta_3) = -3\a \eta_3^3 f\Big(\frac{\eta_1}{\eta_3}\Big) + \b \eta_3, \quad \forall \sum_{i=1}^{3} \eta_i = 0,
\ee
where $f(x) = x^2 + x + \frac{\a-1}{3\a}$. 

Now we choose the same function $\psi$ as defined in (\ref{r-psi}) and consider the range of $\xi$ to be $E_2$ as defined in (\ref{effrange2}). In addition, we decompose $I_1$ and $I_2$ in the same way as in (\ref{I1}) and (\ref{I2}). The main reason of keeping this setup is to ensure that the integral vanishes on $B_1 \times B_1$ and the function $G_2(\xi_1, \xi_2, \xi - \xi_1 - \xi_2, -\xi)$ is of constant size on the rest domain 
\[
	\Omega \triangleq (B_1\times B_2) \cup (B_2\times B_1) \cup (B_2\times B_2).
\] 
As we can see from Section \ref{Subsec, pf of Thm-b}, it is this property that makes it possible to derive a lower bound for $|I_1|$ which dominates $|I_2|$. 

In order to obtain a small threshold for $s$, the above mentioned lower bound for $I_1$ needs to be as small as possible, so another key step is to make the term $|\xi_1 - \xi| / |G_1(\xi_2, \xi - \xi_1 - \xi_2, \xi_1 - \xi)|$ in $I_1$ small. We discuss this issue in three cases based on the value of $\a$.

\begin{itemize}
	\item {\bf Case $\a > 4$}. In this case, $f(x)$ can be rewritten as  
	\be\label{fma}
	f(x) = \Big( x + \frac12 \Big)^2 + m_\a,
	\ee
	where 
	\be\label{ma}
	m_\a = \frac{\a-1}{3\a} - \frac14 = \frac{\a - 4}{12\a} > 0.
	\ee
	So $f$ has a positive lower bound which can be used to ensure $G_1$ to be large. More precisely, we infer from (\ref{gG1}) that
	\be\label{G1p}
	G_1(\eta_1, \eta_2, \eta_3) = -3\a \eta_3 \Big[ \Big(\eta_1 + \frac12 \eta_3\Big)^2 + m_\a \eta_3^2 \Big] + \b \eta_3, \quad \forall\, \sum_{i=1}^{3} \eta_i = 0.
	\ee
	Keeping the constraint $\sum\limits_{i=1}^{3} \eta_i = 0$ in mind, then it is readily seen that whenever $\max\limits_{1\leq i \leq 3} |\eta_i| \gg 1$, we have 
	\be\label{G1s}
		| G_1(\eta_1, \eta_2, \eta_3) | \sim |\eta_3| \sum_{i=1}^{3} \eta_i^2.
	\ee
	This means that the effect of the linear term $\b \eta_3$ is negligible. On the other hand, noticing that $|\xi| \sim N$ on the support $B$ of the function $\wh{\psi} = \wh{\psi}(\xi)$,
	so we have 
	\be\label{large-xi}
		|\xi_1| \sim |\xi_2| \sim |\xi - \xi_1 - \xi_2| \sim N \gg 1
	\ee
	on the support of the term $\wh{\psi}(\xi_1) \wh{\psi}(\xi_2) \wh{\psi}(\xi - \xi_1 - \xi_2)$ in $I_1$. Consequently, 	 
	\be\label{small-ratio-G1}
		\frac{|\xi_1 - \xi|}{ | G_1(\xi_2, \xi - \xi_1 - \xi_2, \xi_1 - \xi) | } \sim \frac{|\xi_1 - \xi|}{N^2 |\xi_1 - \xi|} \sim \frac{1}{N^2}.
	\ee
	Thanks to the above estimate which is much smaller than that in Case 2 in Section \ref{Subsec, pf of Thm-b}, the threshold for $s$ in this case can be lowered to be $-\frac34$.
	
	\item {\bf Case $\a < 0$}. This case is almost identical with the case $\a > 4$ since the constant $m_\a$, as defined in (\ref{ma}), is also a positive constant when $\a < 0$.
	
	\item {\bf Case $\a = 1$}. In this case, $m_\a = -\frac14$ and $f(x) = x(x+1)$. This $f$ is not always positive, so the above strategy does not apply straightforwardly. Meanwhile, the function $G_1$ becomes 
	\be\label{G1-a1}
		G_{1}(\eta_1, \eta_2, \eta_3) = 3\eta_1 \eta_2 \eta_3 + \b\eta_3, \quad \forall\, \sum_{i=1}^{3} \eta_i = 0.
	\ee
	In particular, 
	\[
		G_1(\xi_2, \xi-\xi_1-\xi_2, \xi_1-\xi) = (\xi_1-\xi) \big[3\xi_2(\xi-\xi_1-\xi_2) + \b \big].
	\]
	Due to the observation (\ref{large-xi}) again, we know 
	\[
		\frac{|\xi_1 - \xi|}{ | G_1(\xi_2, \xi - \xi_1 - \xi_2, \xi_1 - \xi) | } \sim \frac{1}{N^2},
	\]
	which matches (\ref{small-ratio-G1}). Therefore, by analogous argument, one can also concludes the threshold for $s$ is $-\frac34$.

\end{itemize}

\appendix

\section{An ill-posedness argument}
As mentioned in the introduction, when $\a = 4$ and $\b = 0$, (\ref{m-MB}) is known to be analytically LWP in $\mathscr{H}^{s}(\m{R})$ for any $s \geq \frac34$. Next, we will show that $\frac34$ is the smallest value for $s$ in order for (\ref{m-MB}) to be at least $C^2$ LWP.

\begin{proposition}\label{Prop, old_ip}
	Let $\a = 4$ and $\b = 0$. Then (\ref{m-MB}) fails to be $C^2$ LWP in $\mathscr{H}^{s}(\m{R})$ for any $s < \frac34$.
\end{proposition}

\begin{proof}
This proof is very similar to that in Section \ref{Subsec, pf of Thm-a} which studied the case when $\a = 4$ and $\b > 0$. In the current case, the argument is actually much simpler since $\b = 0$. Next, we will only highlight the key modification based on the argument in Section \ref{Subsec, pf of Thm-a}. 

Firstly, since $\b = 0$, we will choose $\b_1 = 0$. Consequently, the resonance function in (\ref{res-f1}) becomes 	
\be\label{res-f0}
	\xi_1^3 + \phi^{\a, 0}(\xi - \xi_1) - \phi^{\a, 0}(\xi) = -3 \xi_1 (2\xi - \xi_1)^2.
\ee
Secondly, the key ingredient in the proof is to restrict the resonance function to be comparable to constants. When $\b_1 > 0$ in (\ref{res-size1}), the parameter $\g$ needs to be chosen as $N^{-1}$ in order to achieve this goal. Now since $\b_1 = 0$, if we repeat the proof in Section \ref{Subsec, pf of Thm-a}, then 
\[  
	2\xi - \xi_1 \sim \g \quad\text{and} \quad \text{RHS of (\ref{res-f0})} \sim N \g^2.
\]
Therefore, we only need to choose $\g$ to be $N^{-\frac12}$ to ensure the constant size of the resonance function. 

As a summary, we let $\b_1 = 0$ and choose $\phi$ and $\psi$ such that 
\[
\left\{\begin{aligned} 
	& \wh{\phi}(\xi) = \g_2^{-1/2} N^{-s} \, \m{I}_{[0,\g_2]} (\xi - 2N), \\
	& \wh{\psi}(\xi) = \g_2^{-1/2} N^{-s} \, \m{I}_{[\g_2, 2\g_2]} (\xi + N),
\end{aligned}\right.
\]
where $\g_2 = N^{-1/2}$. Then by following the proof in Section \ref{Subsec, pf of Thm-a}, we conclude that  $s$ is at least $\frac34$ if the solution map is required to be $C^2$.
\end{proof}

\section*{Acknowledgments}



\bigskip

\thanks{(X. Yang) School of Mathematics, Southeast University, Nanjing, Jiangsu 211189, China} 

\thanks{Email: xinyang@seu.edu.cn}

\medskip

\thanks{(S. Li) School of Mathematical Sciences, University of Electronic Science and Technology of China, Chengdu, Sichuan 611731, China}

\thanks{Email: lish@uestc.edu.cn}

\medskip



\thanks{(B.-Y. Zhang) Department of Mathematical Sciences, University of Cincinnati, Cincinnati, OH 45221, USA}

\thanks{Email: zhangb@ucmail.uc.edu}

\end{document}